\documentclass[reqno]{amsart}
\usepackage{amsmath}
\usepackage{amssymb}
\usepackage{amsthm}
\usepackage{amsfonts}
\usepackage{mathtools}
\usepackage{bm}
\usepackage[normalem]{ulem}
\usepackage{xcolor}
\usepackage{fancyhdr}
\usepackage{hyperref}
\usepackage{mdframed}
\usepackage{mathrsfs}
\usepackage{dsfont}
\usepackage{comment}
\usepackage{graphicx,lipsum,wrapfig}
\usepackage[font={tiny}]{caption}
\usepackage[font={scriptsize}]{caption}
\usepackage{todonotes}
\usepackage{ulem}
\usepackage{ stmaryrd }

\theoremstyle{plain}
\newtheorem{thm}{Theorem}[section]
\newtheorem{corollary}[thm]{Corollary}
\newtheorem{lemma}[thm]{Lemma}
\newtheorem{prop}[thm]{Proposition}
\newtheorem*{theorem*}{Theorem}
\theoremstyle{definition}
\newtheorem{rem}[thm]{Remark}
\newtheorem{dfn}[thm]{Definition}

\newtheorem{assumption}{Assumption}
\numberwithin{equation}{section}

\newcommand{\mytilde}{\raise.17ex\hbox{$\scriptstyle\mathtt{\sim}$}}

\newcommand{\Rat}[1]{\ensuremath{\mathbb{R}^{#1}}}

\renewcommand{\P}{\mathcal{P}}

\newcommand{\bE}{\ensuremath{\mathbb{E}}}
\newcommand{\eps}{\ensuremath{\varepsilon}}
\newcommand{\red}[1]{\textcolor{red}{#1}}

\newcommand{\longthmtitle}[1]{\mbox{}{\bf \textit{(#1).}}}

\newcommand{\margind}[1]{\marginpar{\textcolor{blue}{\tiny\ttfamily#1}}}
\newcommand{\marginl}[1]{\marginpar{\textcolor{green!70!black}{\tiny\ttfamily#1}}}

\allowdisplaybreaks

\title[Tractable DRO reformulations over structured ambiguity sets]{Tractable reformulations of DRO problems over structured optimal transport ambiguity sets}

\author{Lotfi M. Chaouach}

\author{Tom Oomen}

\author{Dimitris Boskos}

\thanks{All the authors are with the Delft Center for Systems and Control of TU Delft. Tom Oomen is also with the Department of Mechanical Engineering of TU Eindhoven, \tt{\{L.Chaouach,D.Boskos\}@tudelft.nl,T.A.E.Oomen@tue.nl}.}

\begin{document}

\maketitle

\begin{abstract}
Structuring ambiguity sets in Wasserstein-based distributionally robust optimization (DRO) can improve their statistical properties when the uncertainty consists of multiple independent components. The aim of this paper is to solve stochastic optimization problems with unknown uncertainty when we only have access to a finite set of samples from it. Exploiting strong duality of DRO problems over structured ambiguity sets, we derive tractable reformulations for certain classes of DRO and uncertainty quantification problems. We also derive tractable reformulations for distributionally robust chance-constrained problems. As the complexity of the reformulations may grow exponentially with the number of independent uncertainty components, we employ clustering strategies to obtain informative estimators, which yield problems of manageable complexity. We demonstrate the effectiveness of the theoretical results in a numerical simulation example. 
\end{abstract}

\section{Introduction}
Decisions under uncertainty are widespread in engineering, and the uncertainty often has specific sources and structure.
A typical way to characterize the uncertainty is through probabilistic models. In addition to the range of an uncertain outcome, these models also capture its expected frequency in each region of its range. Stochastic optimization revolves around making decisions in the face of probabilistic uncertainty~\cite{AS-DD-AR:14}. This yields an effective paradigm for system design and operation when the probabilistic model accurately reflects the nature of the uncertainty. However, such probabilistic models are often not available in practice and need to be inferred from a limited amount of data. This may in turn lead to modeling imperfections that can negatively impact design requirements such as efficiency, safety, and reliability. It is therefore essential to seek guarantees for the performance of stochastic optimization problems under uncertainty about their underlying probabilistic model.

For optimization problems with uncertain constraints that need to be met with prescribed probability, also referred to as chance-constrained problems, the scenario approach yields decisions that respect the constraints with distribution-free guarantees~\cite{MCC-SG:18}. The nominal form of this data-driven method robustifies the optimization problem against the worst-case scenario among i.i.d. realizations of the uncertainty. This way, the scenario approach determines tight feasibility regions for chance-constrained problems ~\cite{MCC-SG:11}, which can even be exact~\cite{MCC-SG:08}.
Multiple domains have successfully employed the scenario approach, including robust control~\cite{SG-MCC:13}, model predictive control~\cite{GCC-LF:12}, and interval prediction \cite{SG-MCC-AC:19}. However, its guarantees regarding the satisfaction of probabilistically tight constraints with high confidence, rely on large amounts of samples, which are not always available. In addition, the realizations of the uncertainty may be corrupted by noise, and then it is no longer clear how to retain the probabilistic guarantees of the scenario theory. This motivates the consideration of distributionally robust approaches for data-driven optimization problems. Despite not being fully distribution free, such approaches can provide probabilistic guarantees for any number of samples for general classes of unknown probability distributions.

Distributionally robust optimization (DRO) aims at hedging against distributional ambiguity, which is typical in real-life stochastic uncertainty. To this end, it considers a range of probabilistic models that could likely characterize the uncertainty and determines an optimal solution that is robust against all the distributions in that range. This paradigm has proven to be useful in several applications that require effective handling of data or parameter uncertainty, including machine learning \cite{SS-DK-PME:19}, portfolio optimization \cite{DB-VG-NK:18}, power dispatch \cite{JL-YC-CD-JL-JL:20}, and scheduling \cite{RJ-MR-GX:19}. DRO has also been employed to solve problems in stochastic control, such as distributionally robust linear quadratic regulator problems \cite{IY:21, IT-CDC-CNH:21}, model predictive control algorithms \cite{PC-PP:22, BPGVP-DK-PJG-MM:15}, and more general formulations in distributionally robust dynamic programming \cite{AN-LEG:05}.

A widely used approach is to group plausible probability distributions inside a ball in the Wasserstein metric, which is usually centered at the empirical distribution of the data. This choice enables the designer to tune the radius of the ambiguity ball so that it contains the true probability distribution with prescribed probability~\cite{NF-AG:15} and yields DRO problems that admit tractable reformulations~\cite{PME-DK:17, JB-KM:19,RG-AJK:16}. Decisions of DRO problems over Wasserstein ambiguity balls are also asymptotically consistent~\cite{NF-AG:15}. Namely, as the number of samples goes to infinity, the ambiguity radius can be tuned so that the optimal value and decision of the DRO problem converge to the corresponding quantities of the ideal stochastic optimization problem. In the finite sample regime, 
Wasserstein ambiguity balls that maintain probabilistic guarantees of containing the true distribution suffer from the curse of dimensionality. In particular, their contraction rate with respect to the number of samples scales poorly with the dimension of the uncertainty~\cite{NF-AG:15, LMC-TO-DB:23, LMC-DB-TO:22}.

A recent line of research aims at rendering the decay rate of the ambiguity ball radius independent of the uncertainty dimension by informing it from the associated optimization problem~\cite{JB-YK-KM:16, SS-DK-PME:19, NS-JB-SG-MS:20, RG:20, JB-KM-NS:21}, thereby suppressing the curse of dimensionality. 
Despite these results, the curse of dimensionality still persists when solving multiple distributionally robust optimization problems with the same underlying uncertainty. Some instances of this situation include model predictive control (MPC)~\cite{PC-PP:22, FW-MEV-BH:22, LA-MF-JL-FD:23}, strategy synthesis for Markov decision process~\cite{IG-DB-LL-MM:23}, and in general dynamic programming problems~\cite{IY:21}. In these cases, the reliability of the DRO solution can be assured if the associated ambiguity set contains the data generating distribution with high confidence. Yet, achieving this with traditional Wasserstein ambiguity balls comes with the compromise of their slow shrinkage rate with the number of samples for high-dimensional uncertainty.

For DRO problems where the ambiguity set contains the true distribution with high confidence, \cite{LMC-TO-DB:23} introduced a new class of ambiguity sets that can ameliorate the curse of dimensionality. This is achievable when the uncertainty consists of several independent components. The corresponding ambiguity sets, termed multi-transport hyperrectangles (MTHs), are constructed by imposing multiple optimal transport constraints, one for each component of the uncertainty. These hyperrectangles shrink at much faster rates with the number of samples compared to their monolithic counterparts while maintaining the same probabilistic guarantees of containing the data-generating distribution. This in turn, can facilitate reducing the conservativeness of their associated DRO problems compared to Wasserstein balls. DRO problems over MTHs are also accompanied by finite-dimensional duality results \cite{LMC-TO-DB:23}. 
However, there are yet no systematic reformulations of their dual problems that can be solved by tractable algorithms.  

The aim of this paper is to develop tractable reformulations for several classes of DRO problems over MTHs. To this end, our first contributions exploit DRO duality over hyperrectangles to derive tractable forms of distributionally robust average cost problems, including those with piecewise affine and quadratic costs. We also exploit the results for piecewise affine costs to solve distributionally robust uncertainty quantification problems for general polyhedral sets. Our second contribution is the tractable reformulation of distributionally robust chance-constrained problems over MTHs. To obtain this reformulation, our third contribution is the establishment of basic properties such as weak compactness of MTHs and conditions under which their associated DRO cost is finite. We also generalize a stochastic minimax theorem, which enables us to treat constraints where the uncertainty-dependent argument may become unbounded. Our last contribution revolves around how to cluster the reference distribution (center) of the MTHs to reduce the complexity of their DRO reformulations when the uncertainty consists of many independent components. We also elaborate on tradeoffs between this complexity reduction and the potential inflation of the hyperrectangles to retain their probabilistic guarantees. The tractable reformulations of distributionally robust uncertainty quantification problems appeared in~\cite{LMC-TO-DB:23CDC} but without proofs, which we provide here.

The paper is organized as follows. Section \ref{sec:preleminaries} introduces necessary preliminaries and notation. Section~\ref{sec:problem:formulation} formulates the problem and Section \ref{sec:basic:results} presents basic properties of MTHs. In Section \ref{sec:DRO:reformulations}, we provide tractable reformulations for DRO problems over multi-transport ambiguity hyperrectangles. In Section \ref{sec:uncertainty:quantification}, we specialize these reformulations for distributionally robust uncertainty quantification problems. Section \ref{sec:DRCCP} introduces corresponding reformulations for distributionally robust chance-constrained problems. In Section \ref{sec:numerical:complexity}, we address the computational complexity of the reformulations. We illustrate the results in a simulation example in Section~\ref{sec:simulation:example}. 


\section{Preliminaries and notation} \label{sec:preleminaries}

Throughout this paper, we use the following notation. We denote by $\|\cdot \|$ an arbitrary norm on $\Rat{d}$
and by $\|\cdot\|_p$  the $p$th norm for $p\in[1,\infty]$.  The dual of a norm  $\|\cdot\|$ is denoted by $\|\cdot\|_*$, where $\|z\|_*:=\sup_{\|\xi\|\le1}\langle  z,\xi\rangle$.
We denote by $\mathbb N_{>0}$ the positive integers, by  $\mathbb R_{\ge0}$ and $\mathbb R_{>0}$ the positive and strictly positive real numbers, respectively, and define $\bar{\mathbb R}:=\mathbb R\cup\{-\infty,+\infty\}$.
%
%
For $N\in\mathbb N_{>0}$, we denote $[N]:=\{1,\ldots,N\}$. We denote by $\mathbb S_n$ the set of $n\times n$ real-valued symmetric matrices. Given the vectors $\bm\lambda:=(\lambda_1,\ldots,\lambda_n)\in\mathbb R^n$ and $\bm d:=(d_1,\ldots,d_n)\in\mathbb N^n$ with $d_1+\cdots+d_n=d$, ${\rm diag}^{\bm d}(\bm\lambda)$ refers to the $d\times d$ diagonal matrix that takes the value $\lambda_k$ from the $\ell_{k-1}+1$th to the $\ell_k$th entry of its diagonal, where $\ell_k:=d_1+\cdots+d_k$ for $k\ge 1$ and $\ell_0:=0$. We use the notation $x_+:=\max\{x,0\}$ for a real number $x$. 
%
Given the set $\Xi=\Xi_1\times\ldots\times\Xi_n$ and distinct $k,l\in[n]$, we define the projection operators ${\rm pr}_k:\Xi\to \Xi_k$ and ${\rm pr}_{kl}:\Xi\to \Xi_k\times\Xi_l$  as ${\rm pr}_k(\xi):=\xi_k$ and ${\rm pr}_{kl}(\xi):=(\xi_k,\xi_l)$, respectively, for all $\xi=(\xi_1,\ldots,\xi_n)\in\Xi$. Given also $\bm d:=(d_1,\ldots,d_n)\in\mathbb N^n$, and $d:=d_1+\cdots+d_n$, we define ${\rm pr}_k^{\bm d}:\Xi^d\to\Xi^{d_k}$ by ${\rm pr}_k^{\bm d}(\xi_1,\ldots,\xi_d):=(\xi_{\ell+1},\ldots,\xi_{\ell+d_k})$ with $\ell:=d_1+\cdots+d_{k-1}$ for $k>1$ and $\ell:=0$ for $k=1$.  The conjugate of the function $h:\Rat{d}\to\bar{\mathbb R}$ is defined by $h^*(z):=\sup_{\xi\in\Rat{d}}\{\langle z,\xi\rangle-h(\xi)\}$. Given a set $\Xi\subset\Rat{d}$, its characteristic function is defined by  
$\chi_\Xi(\xi):=0$ if $\xi\in\Xi$ and $\chi_\Xi(\xi):=+\infty$, otherwise, and its support function by $\sigma_\Xi(\xi):= \sup_{\xi\in\Xi}\langle z,\xi \rangle$. The support function of a set is equal to the conjugate of its characteristic function.  For a convex function $f:\Rat{n}\to\Rat{}$, its subdifferential at a point $x\in\Rat{n}$ is the set
\begin{align*}
\partial f(x):=\{z\in\Rat{n}:f(y)-f(x)\ge\langle z,y-x\rangle\;\textup{for all}\;y\in\Rat{n}\}. 
\end{align*}
Given $x,y\in\Rat{d}$, we denote $x\succeq y$ if all components of $x-y$ are positive. Vectors will be interpreted as column vectors in linear algebra operations unless indicated by a transpose. A metric space is called proper if all its closed and bounded subsets are compact.

\textit{Probability theory:} Let $(\Xi,\rho)$ be a Polish space, i.e., a complete and separable space $\Xi$ with metric $\rho$.  We denote by $\mathcal{P}(\Xi)$ the space of probability distributions on $\Xi$ with its Borel $\sigma$-algebra. For any $p\ge1$, $\mathcal P_p(\Xi)$ denotes the set of distributions in $\mathcal P(\Xi)$ with finite $p$th moment. The class $\mathcal G_r^{\rm up}$ consists of all real-valued functions $h$ on $\Xi$ that satisfy the growth bound       
\begin{align}
h(\xi)\le C(1+\rho(\zeta,\xi)^r)\;\textup{for all}\;\xi\in\Xi, \label{eq:h:growth:bound}
\end{align}
for some $C>0$ and $\zeta\in\Xi$. We also denote by $\mathcal G_r$ the class of functions $h$ with both $-h,h\in \mathcal G_r^{\rm up}$. 
%
We refer to the set of discrete distributions supported on $K\in\mathbb N$ atoms in $\Xi$ as $\mathcal P^K(\Xi)$.
Given the measurable spaces $(\Omega,\mathcal F)$ and $(\Omega',\mathcal F')$, a measurable map $\Psi: (\Omega,\mathcal F) \to (\Omega',\mathcal F')$ assigns to each measure $\mu$ in $(\Omega,\mathcal F)$ the pushforward measure $\Psi_\#\mu$ in $(\Omega',\mathcal F')$ defined by $\Psi_\#\mu(B):=\mu(\Psi^{-1}(B))$ for all $B\in\mathcal F'$. We denote by $P\otimes Q$ the product measure of $P$ and $Q$. The Dirac distribution centered at $\xi\in\Xi$ is denoted by $\delta_\xi$. The indicator function $\mathds 1_\Theta$ of a set $\Theta\subset\Xi$ is $\mathds 1_{\Theta}(\xi):=1$ if $\xi\in\Theta$ and $0$ otherwise. For any $p\ge 1$, the $p$th Wasserstein distance between two probability distributions $P,Q \in\mathcal{P}_p(\Xi)$ is defined as
\begin{align*}
    W_{p}(Q,P):=\bigg(\underset{\pi\in \mathcal C(Q,P)}{\mathrm{inf}} \left\{\int_{\Xi\times\Xi} \rho(\zeta,\xi)^p{\rm d}\pi(\zeta,\xi)\right\} \bigg)
\end{align*}
Each $\pi\in\mathcal C(Q,P)$ is a transport plan, a.k.a.  coupling between $P$ and $Q$, i.e., a distribution on $\Xi\times\Xi$ with marginals $P=\rm{pr}_{2\#}(\pi)$ and $Q=\rm{pr}_{1\#}(\pi)$, respectively. 

Given a scalar random variable $\gamma$ with distribution $P$, its \textit{value at risk} is defined as its left-side $1-\alpha$ quantile for $\alpha\in (0,1)$. In particular, if we denote by $F_\gamma$ its cumulative distribution, where $F_\gamma(x) := P(\gamma\le x)$, we have
\begin{align*} 
    {\rm VaR}_{1-\alpha}^P(\gamma) := F_\gamma^{-1}(1-\alpha)= \inf \{\theta\in\Rat{}:F_\gamma(\theta)=1-\alpha \}.
\end{align*}
The \textit{conditional value at risk} $\mathrm{CVaR}_{1-\alpha}^{P}(\gamma)$ of $\gamma$ at level $\alpha$ is the average value of $\gamma$ beyond its corresponding value at risk ${\rm VaR}_{1-\alpha}^P(\gamma)$, namely,
\begin{align*}
   \mathrm{CVaR}_{1-\alpha}^{P}(\gamma):= \mathbb E[\gamma \,|\, \gamma\ge \mathrm{VaR}_{1-\alpha}^P(\gamma)].
\end{align*}
It can be shown (cf.~\cite{TR-SU:00}) that the conditional value at risk is equivalently given by  
\begin{align} \label{CVaR:equiv}
\mathrm{CVaR}_{1-\alpha}^{P}(\gamma)=\inf_{\tau\in\mathbb R}\{\alpha^{-1}\mathbb E_{P}[(\gamma+\tau)_+]-\tau\}.
\end{align}

On any subset $\mathcal A$ of $\P(\Xi)$, the weak topology is defined as the weakest topology that makes all the maps $P\mapsto\bE_P[h]$ continuous for all bounded and continuous functions $h$ on $\Xi$. Since $\P(\Xi)$ is metrizable when $\Xi$ is Polish~\cite[Proposition 7.20]{DB-SES:78}, the notions of weak compactness, continuity, and semicontinuity are equivalent to weak \textit{sequential} compactness, continuity, and semicontinuity, respectively.  

\section{Problem formulation} \label{sec:problem:formulation}

Stochastic optimization is decision-making in the presence of probabilistic uncertainty~\cite{AS-DD-AR:14}. It includes optimization problems of the form 
 \begin{align} \label{stochastic:optimization}
     \inf_{x\in\mathcal X} \mathbb E_{P_\xi}[f(x,\xi)],
 \end{align}
where $f$ is the objective function, $x\in\mathcal{X}$ is the decision variable, and $\xi\in\Xi$ is a random variable with  distribution $P_\xi$.
This problem seeks decisions that are optimal on average. Another instance of how uncertainty affects stochastic decision-making is when some of the problem constraints are no longer hard, but only need to be satisfied with prescribed probability. This yields chance-constrained problems of the form 
\begin{align} \label{eq:CC:problem:intro}
\begin{aligned}
    \inf_{x\in \mathcal{X}} \;&  \langle g, x \rangle \\
    \text{s.t.}\; & P_\xi(f(x,\xi)\le 0)\ge 1-\alpha, 
\end{aligned}
\end{align} 
 where $1-\alpha$ designates a probabilistic threshold with which we require the constraint to hold. A challenge for both problems is that the distribution $P_\xi$ is often unknown and there is only access to a finite number of i.i.d. samples $\xi^1,\ldots,\xi^N$ of the uncertainty $\xi$. Using these samples, it is possible to approximate $P_\xi$ by a data-driven estimator $\widehat P_\xi$, such as the empirical distribution $\widehat P_\xi\equiv P_\xi^N:=\frac{1}{N}\sum_{i=1}^N\delta_{\xi^i}$ of the samples. 
 Choosing the empirical distribution as a proxy for $P_\xi$ yields the so-called Sample Average Approximation (SAA) of \eqref{stochastic:optimization}, which provides reliable decisions for large amounts of samples. However, when the number of samples is limited, the SAA may become a poor surrogate of the original optimization problem since the empirical distribution may deviate significantly from the true distribution. To hedge against this uncertainty about the distribution, we consider the distributionally robust formulation
 \begin{equation}
 \inf_{x\in\mathcal X} \sup_{P\in\mathcal{P}^N} 
     \mathbb{E}_{P}\big[f(x,\xi) \big]. \label{DR problem}
 \end{equation}
  In this distributionally robust optimization (DRO) problem, $\mathcal{P}^N$ is an ambiguity set of distributions that contains plausible models for the true distribution and can be inferred from the collected samples. In an analogous way, taking into account that 
  \begin{align*}
P_\xi(f(x,\xi)\le 0)\ge 1-\alpha\iff P_\xi(f(x,\xi)>0)\le \alpha,
  \end{align*}
  the chance-constrained problem \eqref{eq:CC:problem:intro} is robustified as
  \begin{equation}
  \begin{aligned}
    \inf_{x\in \mathcal{X}}\; &  \langle g, x \rangle \\
    \text{s.t.}\; & \sup_{P\in\mathcal{P}^N}P_\xi(f(x,\xi)>0)\equiv
    \sup_{P\in\mathcal{P}^N}\bE[\mathds{1}_{\{\xi\in\Xi:\;f(x,\xi) > 0 \}}(\xi)]\le \alpha. \label{eq:DR:cc:problem}
  \end{aligned}
  \end{equation} 
  
A popular approach to construct data-driven ambiguity sets is to group all distributions that are $\varepsilon$-close to the empirical distribution $P_\xi^N$ in the Wasserstein metric. This yields the Wasserstein ambiguity ball 
\begin{align*}
\mathcal B_p(P_\xi^N,\varepsilon):=\{P\in\mathcal 
P_p(\Xi):W_p(P_\xi^N,P)\le\varepsilon\},
\end{align*}
with center $P_\xi^N$ and radius $\varepsilon$. Higher values of the exponent $p\ge 1$ yield larger weights to distribution dissimilarities that are farther apart. This is aligned with the fact that Wasserstein distances penalize horizontal variations of the reference distribution and can effectively capture their impact on the value of the optimization problem. Wasserstein ambiguity balls also yield DRO problems with tractable reformulations~\cite{PME-DK:17, JB-KM:19, RG-AJK:16} and enjoy finite-sample guarantees of containing the data-generating distribution~\cite{NF-AG:15}. 
As a result, the optimal cost of \eqref{DR problem} provides an upper bound for the optimal cost of \eqref{stochastic:optimization} with prescribed confidence \cite[Theorem 3.5]{PME-DK:17}. The same conclusion also holds for the robust version \eqref{eq:DR:cc:problem} of the chance-constrained problem  \eqref{eq:CC:problem:intro}. In particular, if $\mathcal P^N$ contains the true distribution with high confidence, then the feasible set of $x$ in  \eqref{eq:DR:cc:problem} is contained in the feasible set of  \eqref{eq:CC:problem:intro} with the same confidence. Therefore the cost of \eqref{eq:DR:cc:problem} bounds the cost of \eqref{eq:CC:problem:intro} with high confidence.    

\subsection{Structured optimal transport ambiguity sets}


Structured ambiguity sets seek to exclude distributions that significantly deviate from the true distribution to reduce the gap between the original stochastic optimization problem and its distributionally robust approximation. 
To obtain such ambiguity set for data-driven problems,  \cite{LMC-TO-DB:23} considers the case where the uncertainty $\xi\equiv (\xi_1,\ldots,\xi_n)\in\Xi\equiv\Xi_1\times\ldots\times\Xi_n$ consists of $n$ independent components $\xi_k$, $k\in[n]$. This implies that $P_\xi$ is necessarily a product measure, i.e., $P_\xi=P_{\xi_1}\otimes\ldots\otimes P_{\xi_n}$. The structured ambiguity set $\mathcal P^N$ is then constructed so that its elements are closer to a product reference measure than in a monolithic ambiguity ball. In particular, given a vector $\bm\varepsilon:=(\varepsilon_1,\ldots,\varepsilon_n)\succeq0$ of transport budgets and a reference distribution $ Q$,~\cite{LMC-TO-DB:23} introduces the multi-transport hyperrectangle (MTH)
 \begin{align}
    \mathcal T_p( Q,\bm\varepsilon):=\big\{\textup{pr}_{2\#}\pi: \; & \pi\in\mathcal P(\Xi\times\Xi),\;\textup{pr}_{1\#}\pi= Q,\nonumber \; \text{and} \\ &\;\int_{\Xi\times\Xi} \rho_k(\zeta_k,\xi_k)^p {\rm d}\pi(\zeta,\xi)\le \varepsilon_k^p  \;\textup{for all}\;  k\in[n] \Big\}.
    \label{transport:hyperrectangle:Tp}
\end{align}
This ambiguity set consists of all distributions that respect a set of mass transport constraints, one for each component of the uncertainty.
For data-driven problems where we have $N$ i.i.d. samples of $\xi$, we select the product empirical distribution 
\begin{align}
     Q\equiv \bm P_\xi^N:= P_{\xi_1}^N\otimes\ldots\otimes P_{\xi_n}^N
    \label{eq: prod empiricals}
\end{align}
 as a reference distribution. Here $P_{\xi_k}^N:=\frac{1}{N}\sum_{i=1}^N\delta_{\xi_k^i}$, $k\in[n]$ are the empirical distributions of the components of $\xi$.
 The MTH $\mathcal T_p(\bm P_\xi^N,\bm\eps)$ shrinks at a favorable rate with the number of samples compared to Wasserstein ambiguity balls, under the premise that it is designed to contain the true distribution with prescribed probability. When $\Xi\subset\Rat{d}$, this ameliorates the curse of dimensionality that Wasserstein balls face with respect to $d$~\cite{NF-AG:15}, especially when the number $n$ of the independent components is large~\cite[Proposition 5.2]{LMC-TO-DB:23}.  
 This ambiguity set design can be generalized to data-driven cases where a different number of realizations is available for each independent component $\xi_k$ of $\xi$, as these realizations may be collected separately or asynchronously. Then each marginal empirical distribution is constructed using a different number of samples. 


\subsection{Tractability}

Both the DRO problem \eqref{DR problem} and the robustified chance-constrained problem \eqref{eq:DR:cc:problem} typically involve solving an infinite-dimensional optimization problem over a space of probability distributions.  
Thus, their numerical solution necessitates their reformulation as finite-dimensional optimization problems. It is further desirable to identify problem classes for which these reformations can be solved by efficient algorithms. 
The first objective, i.e., the derivation of dual finite-dimensional reformulations of such DRO problems over MTHs $\mathcal P^N\equiv\mathcal T_p(\bm P_\xi^N,\bm\eps)$ has been established in our recent work~\cite{LMC-TO-DB:23}. Our goal in this paper is to also address the second objective for DRO problems over MTHs. Namely, to provide tractable reformulations for the corresponding duals of the DRO problems for appropriate problem classes.

Building on reformulations for Wasserstein DRO and chance-constrained problems for suitable classes of costs and constraints, such as piecewise-affine or quadratic~\cite{PME-DK:17, HRA-CA-JL:18, DK-PME-VAN-SAS:19}, \textit{we seek to derive tractable DRO reformulations over MTHs}. To this end, \textit{we also aim to establish fundamental properties of MTHs, such as weak compactness,} which play a key role in obtaining certain reformulations. Our final goal is to address the potential complexity of the derived reformulations, which depends on the number of atoms comprising the product empirical distribution $\bm P_\xi^N$, i.e., the reference distribution of the hyperrectangle. Since the number of these atoms grows rapidly with the number $n$ of the uncertainty components, we seek to determine  \textit{clustering strategies that guarantee  a manageable complexity for DRO reformulations while preserving the probabilistic guarantees characterizing MTHs.}

 \section{Multi-transport hyperrectangle (MTH) DRO: fundamental properties} \label{sec:basic:results} 

In this section, we establish certain fundamental properties of MTHs. They include general weak compactness and duality results, which are utilized in later sections to obtain tractable reformulations for concrete classes of DRO and chance-constrained problems over MTHs.
To derive these results, we assume the following product structure for the Polish space $\Xi$.  

\begin{assumption}
\label{assumption:equivalent:metric}
\longthmtitle{Metric space class} 
Let $\Xi:=\Xi_1\times\cdots\times\Xi_n$ and consider the metric spaces $(\Xi,\rho)$ and $(\Xi_k,\rho_k)$, $k\in[n]$. We assume that: 

\noindent (i) The spaces $(\Xi_k,\rho_k)$, $k\in[n]$ are proper.

\noindent (ii)  The metric $\rho$ on $\Xi$ is equivalent to the product metric $(\sum_{k=1}^n \rho_k^q)^{1/q}$ for some $q\in[1,+\infty]$. 
\end{assumption}

Under this assumption, we establish weak compactness of MTHs.

\begin{thm}\label{thm:weak:compactness}
\longthmtitle{Weak compactness}
Let Assumption \ref{assumption:equivalent:metric} hold. Then, for any reference measure $ Q\in\mathcal P_p(\Xi)$ the MTH $\mathcal T_p( Q,\bm\varepsilon)$ defined in \eqref{transport:hyperrectangle:Tp} is weakly compact.
\end{thm}
This result is used later to establish dual reformulations of distributionally robust chance-constrained problems. The proof is given in the Appendix. 
For the rest of the section, we focus on the inner maximization problem of \eqref{DR problem} associated with \eqref{transport:hyperrectangle:Tp}, which is carried out over an infinite-dimensional space of probability distributions. Namely, we consider the problem 
\begin{align}
    \sup_{P\in\mathcal T_p( Q,\bm\eps)} \mathbb E_P[h(\xi)], \label{eq:inner:maximization}
\end{align}
where we fix the decision variable $x$ in \eqref{DR problem} and denote $h(\xi):=f(x,\xi)$ for notational ease. From a modeling perspective, we assume that the cost in \eqref{eq:inner:maximization} is finite-valued for both the reference and the true distribution. It is therefore meaningful to determine conditions under which this finitness property is also retained for the worst-case distribution from the MTH. These are delineated in the next result. 

\begin{prop} \label{prop:finite:value}

\longthmtitle{Finiteness of the optimal value} 
%
Let Assumption \ref{assumption:equivalent:metric}(ii) hold, assume that
the reference distribution $ Q$ has a finite $p$th moment, and let $\bm\eps:=(\eps_1,\ldots,\eps_n)$ with $\eps_k>0$, for all $k\in[n]$. Then the worst-case expectation \eqref{eq:inner:maximization} is finite if and only if $h\in \mathcal G_p^{\rm up}$.
\end{prop}

\begin{proof}
The proof relies on finding two Wasserstein balls, one that encloses the MTH and one that is enclosed by it, and exploiting validity of the desired result for these balls~\cite[Theorem~2]{MY-DK-WW:21}. To this end, note that by Assumption \ref{assumption:equivalent:metric}(ii) there is some $c>0$ such that $c\rho_k\le\rho$ for every $k\in[n]$ and let $\eps_\star:=c\min\{ \eps_k,\;k\in[n]\}$. For any  $P\in \mathcal B_p( Q,\eps_\star)$, we can pick by \cite[Theorem 4.1]{CV:08} a transport plan $\pi\in\mathcal C( Q,P)$ such that
\begin{align*}
    \int_{\Xi\times\Xi} \rho(\zeta,\xi)^p \pi(\zeta,\xi)\le \eps_\star^p.
\end{align*}
Since $c\rho_k\le\rho$, we get from the definition of $\eps_\star$ that
\begin{align*}
    \int_{\Xi\times\Xi}\rho_k(\zeta_k,\xi_k)^pd\pi(\zeta,\xi)\le \int_{\Xi\times\Xi}\rho(\zeta,\xi)^p/c^pd\pi(\zeta,\xi)\le \eps_\star^p/c^p=\eps_k^p
\end{align*}
for all $k\in[n]$, which yields $\mathcal B_p( Q,\eps_\star)\subset\mathcal T_p( Q,\bm\eps)$. From \cite[Proposition 4.6]{LMC-TO-DB:23} and Assumption~\ref{assumption:equivalent:metric}(ii), we can also select a  sufficiently large ball $\mathcal B_p( Q,\eps^\star)$ that contains $\mathcal T_p( Q,\bm\eps)$. Thus, 
\begin{align}
    \sup_{P\in \mathcal B_p( Q,\eps_\star)} \mathbb E_P[h(\xi)]\le \sup_{P\in \mathcal T_p( Q,\bm\eps)} \mathbb E_P[h(\xi)] \le \sup_{P\in \mathcal B_p( Q,\eps^\star)} \mathbb E_P[h(\xi)] 
\end{align}
and it follows from \cite[Theorem 2]{MY-DK-WW:21} that the optimal value of \eqref{eq:inner:maximization} is $+\infty$ if the conditions in the statement are not met and finite otherwise. This completes the proof.  
\end{proof}

In the next result, we strengthen the conditions of Proposition~\ref{prop:finite:value} to establish that the supremum in~\eqref{eq:inner:maximization} can indeed be attained. These conditions also yield weak continuity of the expected value in \eqref{eq:inner:maximization} with respect to the distributions in the MTH.

\begin{thm}
\label{thm:existence:inner:maximization}
\longthmtitle{Existence of optimal solution \&  continuity over the distribution} 
Let Assumption \ref{assumption:equivalent:metric} hold and assume the reference distribution $ Q$ has a finite $p$th moment. Then:

\noindent (i) If $h\in\mathcal G_r^{\rm up}$ for some $r\in[0,p)$ and $h$ is upper semicontinous, the supremum in \eqref{eq:inner:maximization} is attained.

\noindent (ii) If $h\in\mathcal G_r$ for some $r\in[0,p)$ and $h$ is  continous, the map  $\Psi: \mathcal P_p(\Xi)\to \mathbb R$ with $\Psi(P):= \mathbb E_P[h(\xi)]$ is weakly continuous on $\mathcal T_p( Q,\bm\eps)$.
\end{thm}

We provide the proof in the Appendix. From~\eqref{transport:hyperrectangle:Tp}, one can directly verify that MTHs are convex as they are defined through a finite number of linear constraints in the space of couplings on $\Xi\times\Xi$. This enables the derivation of dual reformulations for DRO problems over MTHs. These equivalent formulations avoid the maximization in the infinite-dimensional space of probability distributions and provide the stepping stone to obtain tractable DRO problems. The following results from \cite{LMC-TO-DB:23} furnishes the dual of \eqref{eq:inner:maximization} for general objective functions.  
\begin{prop} \label{prop:dual}
\longthmtitle{DRO dual over MTHs \cite[Theorem 6.4]{LMC-TO-DB:23}}  
Let $h:\Xi\xrightarrow{}\mathbb R$ be upper semicontinuous with $h\in L^1( Q)$ and consider the problem \eqref{eq:inner:maximization}. Then its dual is given by
\begin{align*}
   \inf_{\bm\lambda\succeq0}\Big\{\langle\bm\lambda, \bm\epsilon\rangle+\int_{\Xi}\sup_{\xi\in\Xi}\Big\{h(\xi)-\sum_{k=1}^n \lambda_k \rho_k(\zeta_k,\xi_k)^p\Big\}{\rm d} Q(\zeta)\Big\}, 
\end{align*}
where $\bm\lambda:=(\lambda_1,\ldots,\lambda_n)$ and $\bm\epsilon:=(\eps_1^p,\ldots,\eps_n^p)$, and there exists $\bm\lambda$ such that the infimum is attained. When $ Q$ is the discrete distribution  $\sum_{l=1}^M\vartheta_l\delta_{\xi^l}$, the dual problem becomes
\begin{align*}
   \inf_{\bm\lambda\succeq0}\Big\{\langle\bm\lambda,\bm\epsilon\rangle+\frac{1}{M}\sum_{l=1}^M\sup_{\xi\in\Xi}\Big\{h(\xi)-\sum_{k=1}^n \lambda_k \rho_k(\xi_{k}^l,\xi_k)^p\Big\}\Big\}.
\end{align*}
\end{prop}
%



The expression of the dual problem over discrete distributions, such as the empirical or the product empirical distribution, hints at the fact that tractable reformulations hinge on explicit ways to evaluate the suprema therein with respect to $\xi$. This will indeed be the case for all the specific problem classes that we tackle in the later sections.

\section{Tractable reformulations for DRO problems associated with MTHs} \label{sec:DRO:reformulations}
In this section, we consider MTHs \eqref{transport:hyperrectangle:Tp} that are centered at a discrete distribution $ Q$ and exploit Theorem \ref{prop:dual} to derive tractable reformulations of the DRO problem \eqref{eq:inner:maximization}. We make the following assumption regarding the reference distribution $ Q$. 

\begin{assumption} \label{assumption:P_hat}
\longthmtitle{Reference distribution} 
The reference distribution $ Q$ is discrete and consists of $M$ atoms $\xi^1,\ldots,\xi^M$ with masses $\vartheta_1,\ldots,\vartheta_M$, respectively.  
\end{assumption}

Here, we mainly build on the convex reformulations from \cite{PME-DK:17}, which considers ambiguity balls using the 1-Wasserstein distance in $\Rat{d}$. We therefore focus on  $\Rat{d}$-valued random variables  $\xi=(\xi_1,\ldots,\xi_n)\in\Xi\subset\Rat{d}\equiv\Rat{d_1}\times\ldots\times\Rat{d_n}$.  
To transform the optimization problem \eqref{eq:inner:maximization} into a tractable convex program, we also need to make certain assumptions about the domain of the uncertainty $\Xi$ and the objective function $h$. 
As in \cite{PME-DK:17}, we assume that the objective function can be expressed as the point-wise maximum of a finite family of concave functions.

\begin{assumption}
\longthmtitle{Convex decomposition}
The set $\Xi\subset\mathbb R^d$ is convex and closed. Further,  the objective function $h:\Rat{d}\to\Rat{}$ is real-valued and has the form $h(\xi)=\max_{j\in[m]}h_j(\xi)$, for some $m\in\mathbb N_{>0}$, where the functions $-h_j:\Rat{d}\to\bar{\Rat{}}$, $j\in[m]$ are convex, proper, lower semi-continuous, and not identically $+\infty$ on $\Xi$ (although some of the functions $h_j$ may attain the value $-\infty$ individually, we assume their collective maximum to be always finite).  
\label{assumption:hj:concave}
\end{assumption}
 In the next result, we exploit Assumptions \ref{assumption:P_hat} and \ref{assumption:hj:concave} to derive a finite-dimensional convex reformulation for \eqref{eq:inner:maximization}. All the proofs of the section are provided in Appendix~\ref{sec:appendix:A}.

\begin{prop}
\longthmtitle{Finite convex program} Consider the optimization problem \eqref{eq:inner:maximization} with $p=1$ and some $\bm\varepsilon\succeq0$. If $Q$ and $h$ satisfy Assumptions~\ref{assumption:P_hat} and~\ref{assumption:hj:concave}, respectively, then the worst-case expectation \eqref{eq:inner:maximization} can be evaluated by solving the finite convex program
\begin{equation} \label{eq: finite convex program}
\begin{aligned}
    \inf_{\underset{l\in[M], j\in[m]}{\bm\lambda,s_{l},z_{lj},\upsilon_{lj}}} &\;\langle\boldsymbol{\lambda, \varepsilon}\rangle+\sum_{l=1}^{M}\vartheta_l s_l &&  \\
    {\rm s.t.}\quad & [-h_j]^*(z_{lj}-\upsilon_{lj})+\sigma_{\Xi}(\upsilon_{lj})-\langle z_{lj},\xi^l\rangle\le s_l\;&& \quad l\in[M],\;j\in[m] \\
    &\|\textup{pr}_k^{\bm d}(z_{lj})\|_*\le\lambda_k\;  && \quad l\in[M],\;j\in[m], \;k\in[n]. 
    \end{aligned}
\end{equation}
Here, $\bm\lambda:=(\lambda_1,\ldots,\lambda_n)$, $\bm d:=(d_1,\ldots, d_n)$ and $\sigma_\Xi$ is the support function of $\Xi$. 
\label{prop:finite:convex:program}
\end{prop}

Next, we specialize this result to piecewise affine objective functions. These form a rather general function class, as they can approximate any nonlinear function of bounded variation with arbitrary accuracy~\cite{JK-FR:16}.

\begin{prop}
\longthmtitle{Piecewise affine objective functions} \label{prop:Paffine:reformulation}
Assume $p=1$ and suppose that the domain of the uncertainty set is the polyhedral set $\Xi:=\{\xi\in\mathbb R^d:\;C\xi\preceq f\}$, where $C\in\Rat{r\times d}$, and that  $h_j(\xi):=\langle\alpha_j,\xi\rangle+b_j$, $j\in[m]$. 

\noindent (i) If $h_j(\xi)=\max_{j\in[m]}h_j(\xi)$, then \eqref{eq:inner:maximization} can be evaluated by solving
\begin{align}
\begin{aligned}
\inf_{\underset{l\in[M], j\in[m]}{\bm\lambda,s_{l},\gamma_{lj}}}\; &\langle\boldsymbol{\lambda, \varepsilon}\rangle+\sum_{l=1}^{M}\vartheta_l s_l && \\
{\rm s.t.}\quad\; & b_j+\langle \alpha_j,{\xi}^l\rangle+\langle \gamma_{lj},f-C {\xi}^l\rangle\le s_l && \quad l\in[M],\;j\in[m] \\
&\left\|\textup{pr}_k^{\bm d}(C^\top\gamma_{lj}-\alpha_j)\right\|_*\le\lambda_k  && \quad l\in[M],\;j\in[m], \;k\in[n] \\
&\gamma_{lj}\succeq0  &&  \quad l\in[M],\;j\in[m].
\end{aligned}
\end{align}

\noindent (ii) If $h_j(\xi)=\min_{j\in[m]}h_j(\xi)$, then \eqref{eq:inner:maximization} can be evaluated by solving
\begin{align}
\begin{aligned}
\inf_{\underset{l\in[M]}{\bm\lambda,s_{l},\gamma_l,\theta_l}} &\langle\boldsymbol{\lambda, \varepsilon}\rangle+\sum_{l=1}^{M}\vartheta_l s_l &&  \\
{\rm s.t. }\quad & \langle \theta_l,b+A\xi^l\rangle+\langle \gamma_l,f-C{\xi}^l\rangle\le s_l\;&& \quad l\in[M] \\
&\left\|{\rm pr}_k^{\bm d} (C^\top\gamma_l-A^\top\theta_l)\right\|_*\le\lambda_k  && \quad l\in[M], \; k\in[n] \\
&\langle \theta_l,\bm 1\rangle=1 && \quad l\in[M] \\
&\gamma_l\succeq0, \; \theta_l\succeq0 &&  \quad l\in[M].
\end{aligned}
\end{align}
In these reformulations, $A\in\Rat{m\times d}$ is the matrix with rows $\alpha_j^\top$,  $b$ is the column vector with entries $b_j$, $\bm 1$ is a vector of ones, $\bm\lambda:=(\lambda_1,\ldots,\lambda_n)$, and $\bm d:=(d_1,\ldots, d_n)$.
\end{prop}

Based on~\cite{DK-PME-VAN-SAS:19} we also consider quadratic objective functions and derive tractable reformulation when $p=2$, $\Xi=\Rat{d}$ and $\|\cdot\|$ is the Euclidean norm.

\begin{prop} \longthmtitle{Indefinite quadratic objective functions} \label{prop:quad:reformulation}
    Assume $\Xi=\mathbb R^d$ is equipped with the Euclidean norm, and consider the indefinite quadratic loss function $h(\xi)=\xi^\top \mathcal Q\xi+2q^\top \xi$ with $\mathcal Q\in \mathbb S_d$ and $q\in\mathbb R^d$. Then, under Assumption \ref{assumption:P_hat}, the worst-case expectation \eqref{eq:inner:maximization} with $p=2$, can be evaluated by solving the tractable semi-definite program
    \begin{align}
    \begin{aligned}
        \inf_{\underset{l\in[M]}{\bm\lambda\succeq0, s_l\ge0}}& \langle \bm\lambda,\bm\epsilon\rangle + \sum_{l=1}^{M} \vartheta_l s_l \\ 
               {\rm s.t.} \quad   &\begin{bmatrix}
                    {\rm diag}^{\bm d}(\bm\lambda)-\mathcal Q & q+{\rm diag}^{\bm d}(\bm\lambda)\xi^l \\
                    q^\top+{\xi^l}^\top {\rm diag}^{\bm d}(\bm\lambda) & s_l+{\xi^l}^\top {\rm diag}^{\bm d}(\bm\lambda) \xi^l
                \end{bmatrix}\succeq0 && \qquad \qquad l\in[M] 
    \end{aligned}
    \end{align}
    where $\bm\epsilon:=(\eps_1^2,\ldots,\eps_n^2)$, $\bm\lambda:=(\lambda_1,\ldots,\lambda_n)$, and $\bm d:=(d_1,\ldots, d_n)$. 
\end{prop}


\section{Uncertainty quantification using structured ambiguity sets }
\label{sec:uncertainty:quantification}

 In this section, we develop tractable reformulations to solve uncertainty quantification problems over MTHs. Such problems can be of important practical interest in applications where we seek to assess whether a physical or engineered system is safe or not. To this end, we determine a set of constraints that must be satisfied by the uncertain state of the system with a specified probability so that the system can be qualified as safe or unsafe.

We focus on scenarios where the distribution of the state is unknown and we can only exploit samples to infer worst-case probabilities of being either safe or unsafe with high confidence. To achieve this, we exploit the reformulations from \cite[Corollary 5.3]{PME-DK:17} that allow determining worst- and best-case probabilities for a random variable $\xi$ to belong to a polyhedral safe set $\mathbb A$ or not and extend it to the more general case where the safe set is the union of polyhedral sets. Specifically, we solve the problem 
\begin{align}
    \sup_{P\in\mathcal T_1( Q,\bm\varepsilon)} P\left(\xi\in\mathbb A\right),
    \label{eq: safety probability}
\end{align}
where $\mathbb A:=\cup_{j=1}^m\mathbb A_j$ and $\mathbb A_j$ are polyhedral sets for all $j\in[m]$. 
\begin{thm} \label{thm:uncertainty:quantification}
    \longthmtitle{Uncertainty quantification for unions of polyhedral sets} Let the support of $P_\xi$ be the nonempty polyhedral set $\Xi:=\{\xi\in\mathbb R^d: C\xi\preceq f\}$ and assume that each set $\mathbb A_j:=\{\xi\in\mathbb R^d:A_j\xi\preceq b_j\}$, $j\in [m]$ has nonempty intersection with $\Xi$. Then, under Assumption~\ref{assumption:P_hat}, the probability \eqref{eq: safety probability} can be evaluated by solving the convex program 
    \begin{equation}     
    \begin{aligned}
    \inf_{\underset{l\in[M], j\in[m]}{\bm\lambda,s_l,\gamma_{lj},\theta_{lj}}} \; & \langle\bm\lambda,\bm\varepsilon\rangle+\sum_{l=1}^{M} \vartheta_l s_l \label{eq: best-case probability} \\
        {\rm s.t.}\quad\; & 1+\langle \theta_{lj},b_j-A_j {\xi}^l\rangle+\langle \gamma_{lj},f-C {\xi}^l\rangle\le s_l && l\in [M] \\
        &\| \textup{pr}_k^{\bm d} (A_j^\top \theta_{lj}+C^\top\gamma_{lj})\|_*\le\lambda_k && l\in [M],\; j\in[m],\; k\in [n] \\
        & \gamma_{lj}\succeq0, \; \theta_{lj}\succeq0,\; s_l\ge0  && l\in [M],\; j\in[m] 
       \end{aligned}
    \end{equation} 
    where $\bm\lambda:=(\lambda_1,\ldots,\lambda_n)$ and $\bm d:=(d_1,\ldots,d_n)$. 
\end{thm}
\begin{proof}
To prove the theorem, we exploit the fact that
\begin{align*}
     \sup_{P\in\mathcal T_1( Q,\bm\varepsilon)} P\left(\xi\in\mathbb A\right)=
      \sup_{P\in\mathcal T_1( Q,\bm\varepsilon)} \mathbb E_P[\mathds 1_{\mathbb A}]. 
\end{align*}
and use Proposition \ref{prop:finite:convex:program} with appropriate loss functions $h_j$ to recover the indicator function $\mathds 1_{\mathbb A}$. To this end, let $h_j:=1-\chi_{\mathbb A_j}$ for each $j\in[m]$, set $h_{m+1}:=0$, and note that
\begin{align*}
    \mathds 1_{\mathbb A}(\xi)&=\mathds 1_{\mathbb A_1\cup\cdots\cup\mathbb A_m}(\xi)=\max\{h_1(\xi),\ldots,h_{m+1}(\xi)\}=\max\{1-\chi_{\mathbb A_1}(\xi),\ldots,1-\chi_{\mathbb A_m}(\xi),0\}
\end{align*}
for all $\xi\in\Rat{d}$. The conjugates of the first $m$ functions inside the max are 
\begin{align}
    [-h_j]^*(z)=\sup_{\xi\in\mathbb A_j}\;\langle z,\xi \rangle+1 = \begin{cases}
        \underset{\xi}{\sup}\;\langle z,\xi\rangle+1\\
        \text{s.t. } A_j\xi\preceq b_j
    \end{cases}
    = \begin{cases}
        \underset{\theta\succeq0}{\inf}\;\langle \theta,b_j \rangle+1\\
        \text{s.t. } A_j^\top\theta= z.
    \end{cases} \label{eq: conjugate of hj}
\end{align}
Here the last equality follows from the assumption that $\mathbb A_j$ is nonempty and linear programming duality. 
We also express the support function of $\Xi$ as  
\begin{align}
    \sigma_\Xi(\upsilon)=\begin{cases}
\underset{\xi}{\sup} \;\langle \upsilon,\xi\rangle \\
\text{s.t.}\; C\xi\preceq f
    \end{cases}
    = \begin{cases}
       \underset{\gamma\succeq0}{\inf}\;\langle \gamma,f \rangle \\
       \text{s.t.}\; C^\top \gamma=\upsilon,
    \end{cases}
    \label{eq: support function Xi}
\end{align}
where the last equality follows from linear programming duality. 

Substituting the expressions in \eqref{eq: conjugate of hj} and \eqref{eq: support function Xi} into the first set of constrainst in \eqref{eq: finite convex program}, we get  for each $l\in[M]$ and $j\in[m]$ the constraints 
\begin{align*}
    [-h_j]^*(z_{lj}& -\upsilon_{lj})+\sigma_{\Xi}(\upsilon_{lj})-\langle z_{lj},{\xi}^l\rangle \equiv 1+\langle \theta_{lj},b_j\rangle+\langle \gamma_{lj},f\rangle- \langle z_{lj},\xi^l \rangle\le s_l 
\end{align*}
and $C^\top\gamma_{lj}=\upsilon_{lj}$, $
A_j^\top\theta_{lj} = z_{lj}-\upsilon_{lj}$, $\gamma_{lj}\succeq0$, $\theta_{lj}\succeq0$, which are equivalent to $ z_{lj}= A_j^\top\theta_{lj}+C^\top\gamma_{lj}$, $\gamma_{lj}\succeq0$, $\theta_{lj}\succeq0$. Substituting also the last expression for $z_{lj}$ in the  above inequality constraint and the second set of constraints in \eqref{eq: finite convex program}, yields for each $l\in[M]$, $k\in[n]$, and $j\in[m]$ the equivalent set of constraints  
\begin{align*}
1 +\langle \theta_{lj},b_j\rangle+\langle \gamma_{lj},f \rangle-\langle A_j^\top\theta_{lj}+C^\top\gamma_{lj},\xi^l \rangle  & \equiv 1+ \langle \theta_{lj}, b_j-A_j\xi^l\rangle + \langle \gamma_{lj},f-C\xi^l \rangle\le s_l \\
\left\| \textup{pr}_k^{\bm d}(A_j^\top\theta_{lj}+C^\top\gamma_{lj})\right\|_* &  \le\lambda_k,\; \gamma_{lj}\succeq0,\; \theta_{lj} \succeq0.
\end{align*}
%
%
Similarly, we find that $[-h_{m+1}]^*(z)=0$ if $z=0$ and $+\infty$ otherwise. Exploiting \eqref{eq: support function Xi} and using similar steps as above, for each $l$ and $k$, the corresponding constraints in \eqref{eq: finite convex program} become 
 %
\begin{align*}
    \inf_{\gamma_{l,m+1}\succeq 0} \langle \gamma_{l,m+1}, f - C\xi^l\rangle  &\le s_l \\
    \left\| {\rm pr}_k^{\bm d}(C^\top \gamma_{l,m+1})\right\|_* &\le \lambda_k, 
\end{align*}
%
Since $f-C\xi^l\succeq0$, the expression on the left-hand side of the first constraint attains its infimum for $\gamma_{l,m+1}=0$ while automatically satisfying the second constraint. Hence, we get the equivalent constraints  $s_l\ge0$ for all $l\in[M]$. This establishes the desired result.
\end{proof}
Note that besides generalizing \cite[Corollary 5.3(ii)]{PME-DK:17} from Wasserstein ambiguity balls to MTHs, Theorem \ref{thm:uncertainty:quantification} also generalizes the evaluation of worst/best case probability from polyhedral sets to the union of polyhedral sets. 

Instead of assessing how safe the set $\mathbb A$ can be, we may also want to know with high confidence how likely it is for  $\xi$ to lie outside $\mathbb A$. This yields the uncertainty quantification problem
\begin{align}
        \sup_{P\in\mathcal T_1( Q,\bm\varepsilon)} P \left(\xi\notin\mathbb A \right),
    \label{eq: unsafety probability}
\end{align}
where $\mathbb A:=\cup_{j=1}^m \mathbb A_j$ and each $\mathbb A_j$ is a polyhedral set. The next result evaluates this worst-case probability. 
\begin{corollary} 
    \longthmtitle{Uncertainty quantification for complements of unions of polyhedral sets} Assume that $P_\xi$ is supported on the nonempty polyhedral set  $\Xi:=\{\xi\in\mathbb R^d: C\xi\preceq f\}$ and let $\mathbb A_j:=\{\xi\in\mathbb R^d:\langle a_j^l,\xi \rangle< b_j^l\;\textup{for all}\; l\in[\alpha_j]\}$, $j\in[m]$. Then, under Assumption \ref{assumption:P_hat}, the value of the program
    \begin{align*}
         \inf_{\underset{l\in[M], q\in Q}{\bm\lambda,s_l,\gamma_{lq},\theta_{lq}}}\; & \langle\bm\lambda,\bm\varepsilon\rangle+\sum_{l=1}^{M} \vartheta_l s_l \\
        {\rm s.t.}\quad \; & 1-\langle \theta_{lq},b_{q}-A_{q}  \xi^l \rangle + \langle \gamma_{lq},f -C \xi^l\rangle \le s_l && l\in[M],\; q\in Q  \\
            &\left\| \textup{pr}_k^{\bm d}\left(C^\top\gamma_{lq}-A_ q^\top\theta_{lq}\right) \right\|_*\le\lambda_k && l\in[M],\; q\in Q,\; k\in[n]  \\
            & \gamma_{lq}\succeq0,\;\theta_{lq}\succeq0, \; s_l\ge0 && l\in[M],\; q\in Q,\; k\in[n] 
    \end{align*}
    is equal to the probability \eqref{eq: unsafety probability}. Here $\bm\lambda:=(\lambda_1,\ldots,\lambda_n)$, $\bm d:=(d_1,\ldots,d_n)$ and for any $q:=(q_1,\ldots,q_m)\in\prod_{j=1}^m[\alpha_j]$, $A_q\in \mathbb R^{m\times d}$ is the matrix formed by concatenating the row vectors  $(a_j^{q_j})^\top$ and $b_q:=(b_1^{q_1},\ldots, b_m^{q_m})$. The set $Q$ comprises all indices  $q\in\prod_{j=1}^m[\alpha_j]$ for which the sets $\{\xi\in\mathbb R^d: A_q\xi \succeq b_q\}$ have nonempty intersection with $\Xi$. \label{cor: worst-case probability}
\end{corollary}
\begin{proof}
    Since each polyhedral set $\mathbb A_j$ is the intersection of the  half-spaces $\mathcal A_j^{q_{j}}:= \{\xi\in\mathbb R^d: \langle a_j^{q_j}, \xi \rangle<b_j^{q_j}\}$, we have that $\mathbb A_j= \cap_{q_j=1}^{\alpha_j}\mathcal A_j^{q_{j}}$. Taking also into account that 
    \begin{align*}
       \sup_{P\in\mathcal T_1( Q,\bm\varepsilon)} P\left(\xi\notin\mathbb A\right) = \sup_{P\in\mathcal T_1( Q,\bm\varepsilon)} P\left(\xi \in\mathbb A^c\right)
    \end{align*}
    and that $\mathbb A=\cup_{j=1}^m \mathbb A_j$, we get
    \begin{align*}
        \mathbb A^c=\cap_{j=1}^m {\mathbb A_j}^c  
                   =\cap_{j=1}^m (\cap_{q_j=1}^{\alpha_j}\mathcal A_j^{q_{j}})^c & = \cap_{j=1}^m \cup_{q_j=1}^{\alpha_j}{(\mathcal A_j^{q_{j}})}^c =\cup_{(q_1,\ldots,q_m)\in [\alpha_1]\times\cdots\times[\alpha_m]}\cap_{j=1}^m ({\mathcal A_j^{q_{j}}})^c,
    \end{align*}
    which follows from De Morgan's law and the distributivity between unions and intersections.  
    Denoting $\mathbb B_q:=\cap_{j=1}^m {(\mathcal A_j^{q_{j}})}^c$, we deduce that $\mathbb A^c$ is the union of the sets $\mathbb B_q$ where $q:=(q_1,\ldots,q_m)\in [\alpha_1]\times\cdots\times[\alpha_m]$. These sets are equivalently expressed as
    \begin{align}
        \mathbb B_q =\{\xi\in\mathbb R^d: \langle a_j^{q_j} , \xi \rangle \ge b_j^{q_j}\;
        \textup{for all}\;j\in[m]\} & =\{ \xi\in\mathbb R^d: A_q\xi\succeq b_q  \}\nonumber \\
        &= \{ \xi\in\mathbb R^d: -A_q\xi\preceq\bm -b_q  \}.
        \label{eq: proof 2 development 2}
    \end{align}
    Keeping only the indexes $q$ for which $\mathbb B_q$ has a nonempty intersection with $\Xi$, and invoking Theorem \ref{thm:uncertainty:quantification}, we obtain the worst-case probability  \eqref{eq: unsafety probability} by replacing $A_j$ and $b_j$ in the first and second set of constraints in  \eqref{eq: best-case probability} with $-A_q$ and $-b_q$ from \eqref{eq: proof 2 development 2}, respectively.  
\end{proof}



\section{Distributionally robust chance-constrained problems} \label{sec:DRCCP}

In this section, we generalize tractable reformulations for chance-constrained problems over Wasserstein ambiguity balls to MTHs. We consider optimization problems that may have multiple chance constraints (see e.g., \cite{GS-LF-MM:13,JKS:70}). Such problems are typical in multistage decision making \cite[Chapter 3]{AS-DD-AR:14}, including stochastic model predictive control \cite{LH-KW-MZ:20,MF-MM-FD:24,AM:16}. Introducing several chance constraints provides the freedom to assign higher tolerance levels to design requirements that are softer than others. For data-driven problems, verifying these constraints individually with high confidence may require accurate knowledge of how the probability mass is distributed across different regions of the uncertainty domain. Hedging against such distribution imperfections can be addressed by considering an ambiguity set that contains the data-generating distribution with high confidence. This further motivates our consideration of MTHs, since they share refined probabilistic guarantees of containing the true distribution. On the other hand, the alternative of grouping all individual chance constraints into a single probabilistic constraint may result in overly conservative decisions or even infeasibility, especially if they are violated for disjoint events.

We build on the approach in \cite{HRA-CA-JL:19}, which provides reformulations of problems with a single distributionally robust chance constraint over Wasserstein ambiguity balls. Here, we are interested in solving problems of the form
\begin{equation}
\begin{aligned} \label{eq:CC:problem}
    \inf_{x\in \mathcal{X}} \; & \langle g, x \rangle && \\
    \text{s.t.}\; & P_\xi(f_i(x,\xi)\le 0)\ge 1-\alpha_i && i\in[I],
\end{aligned}
\end{equation}
where $\mathcal X\subset\Rat{\ell}$, $\xi\in\Xi\subset\Rat{d}\equiv\Rat{d_1}\times\cdots\times\Rat{d_n}$ is a random variable, $g\in\mathbb R^\ell$, $I\in\mathbb N$, and each  $f_i:\mathcal X\times\Xi\to\Rat{}$, $i\in[I]$ designates a chance constraint that needs to be fulfilled with probability at least $1-\alpha_i\in (0,1)$. We further assume that $\Xi$ is closed and $\mathcal X$ is closed and convex. 
%
%
Using the value at risk (see Section~\ref{sec:preleminaries}), we can compactly write  \eqref{eq:CC:problem} as 
\begin{equation} \label{eq:VaR:problem}
\begin{aligned}
    \inf_{x\in\mathcal X}\; & \langle g,x\rangle && \\
        \text{s.t.} \; &  {\rm VaR}_{1-\alpha_i}^{P_\xi}(f_i(x,\xi))\le0 && i\in[I].
\end{aligned}
\end{equation}
This optimization problem may be non-convex, even when the functions $f_i$ are convex. To address this issue, in analogy to \cite{HRA-CA-JL:19}, we approximate the chance constraints using the conditional value at risk (CVaR) of $f(x,\xi)$.
From the definition of the CVaR, we have  
\begin{align*}
\mathrm{CVaR}_{1-\alpha_i}^{P_\xi}(f(x,\xi))\le 0 \implies  {\rm VaR}_{1-\alpha_i}^{P_\xi}(f(x,\xi))\le 0.
\end{align*}
Therefore, using the equivalent characterization \eqref{CVaR:equiv} of the CVaR, we can approximate \eqref{eq:VaR:problem} by the CVaR-constrained  problem
\begin{equation}
    \begin{aligned}
        \begin{aligned}
   \inf_{x\in\mathcal X} \; & \langle g, x \rangle && \\
   {\rm s.t.}\; & \inf_{\tau\in\mathbb R}\mathbb E_{P_\xi}[(f_i(x,\xi)+\tau)_+-\tau\alpha_i]\le 0 && i\in[I]. \end{aligned}
    \end{aligned} \label{eq:CVaR:C:problem}
\end{equation}

Our goal is to provide tractable reformulations for a distributionally robust version of this problem.
We will build on the reformulations of the previous section to solve it when the distribution $P_\xi$ belongs to the MTH $\mathcal T_p( Q,\bm\eps)$. We are interested in solving the distributionally robust chance-constrained problem 
\begin{align}
\begin{aligned}
  \inf_{x\in\mathcal X} \; & \langle g, x \rangle && \\
   {\rm s.t.}  \;  & \sup_{P\in \mathcal T_p( Q,\bm\varepsilon)} \inf_{\tau\in\mathbb R}\mathbb E_{P}[(f_i(x,\xi)+\tau)_+-\tau\alpha_i]\le 0 && i\in[I], \end{aligned}
    \label{eq:DRCCP:CVaR}
\end{align}
which represents a robustified version of \eqref{eq:CVaR:C:problem} against discrepancies in the reference distribution.
Tractable reformulations for this distributionally robust chance-constrained problem rely on obtaining a tractable expression for each CVaR constraint.  We make the following assumption regarding the functions $f_i$, $i\in[I]$. 
\begin{assumption}
 \longthmtitle{Constraint function class} \label{assumption:convex:concave}
 The functions $f_i:\mathcal X\times\Xi\to \Rat{}$ satisfy the following properties: \\
\noindent (i) For each $i\in[I]$ and $\xi\in\Xi$, the function $x\to f_i(x,\xi)$ is convex; \\
\noindent (ii) There exists $r\in[0,p)$ such that for each $i\in[I]$ and $x\in\mathcal X$, the function  $\xi\to f_i(x,\xi)$ is continuous and belongs to $\mathcal G_r$. 
\end{assumption}

Suppressing for the moment the subscript $i$ of the functions in the constraints and their satisfaction probabilities, we seek a tractable characterization for the feasible set 
\begin{align}
      \Big\{ x\in\mathcal X: \sup_{P\in\mathcal T( Q,\bm\eps)} \inf_{\tau\in\mathbb R}\mathbb E_{P}[(f(x,\xi)+\tau)_+-\tau\alpha]\le 0 \Big\}. \label{eq:DR:CVaR:constraints:1}
\end{align}
of a single chance-constraint where the function $f$ satisfies Assumption~\ref{assumption:convex:concave}. 

Determining tractable reformulations for \eqref{eq:DR:CVaR:constraints:1} hinges on interchanging the order between the sup and inf of the distributionally robust constraint. This possibility is guaranteed by the following result, which is based on the fact that MTHs are weakly compact. This result extends \cite[Lemma I.V.2]{HRA-CA-JL:18} to the case where the ambiguity set is the MTH \eqref{transport:hyperrectangle:Tp}. It also entails the consideration of a broader class of constraints, since in \cite{HRA-CA-JL:18} the mappings $\xi\mapsto f(x,\xi)$ are assumed to be bounded. Here, instead, they only need to respect the growth condition of Assumption~\ref{assumption:convex:concave}(ii).

\begin{lemma} \label{lemma:mini-max}
    \longthmtitle{Min-max equality for the $\bm{{\rm CVaR}}$ over MTHs}
    Suppose that Assumption \ref{assumption:convex:concave} holds. Then:
    
    \noindent (i) For all $x\in\mathcal{X}$,
    \begin{align}
        \sup_{P\in\mathcal T_p( Q,\bm\varepsilon)}\inf_{\tau\in\mathbb R} \mathbb E_{P}[(f(x,\xi)+\tau)_+-\tau\alpha]=\inf_{\tau\in\mathbb R} \sup_{P\in\mathcal T_p( Q,\bm\varepsilon)} \mathbb E_{P}[(f(x,\xi)+\tau)_+-\tau\alpha]; \label{eq:min-max}
    \end{align}
    \noindent (ii) The infimum on the right-hand-side of \eqref{eq:min-max} with  respect to $\tau$ is attained.
\end{lemma}

The proofs of this lemma and the other results of this section are given in the Appendix. Next, we exploit DRO duality for MTHs to derive a tractable characterization of the feasible set \eqref{eq:DR:CVaR:constraints:1}.  The following result characterizes this set through a finite number of convex constraints.
\begin{prop} \longthmtitle{Equivalent characterization of the feasible set  \eqref{eq:DR:CVaR:constraints:1}}  \label{prop:reformulations:CVaR:convex:functions}
    Under Assumptions \ref{assumption:P_hat} and \ref{assumption:convex:concave}, \eqref{eq:DR:CVaR:constraints:1} consists of all $x\in\mathcal X$ for which the convex constraints 
    \begin{align*}
    \begin{aligned}
        &\langle\boldsymbol{\lambda, \varepsilon}\rangle+ \sum_{l=1}^{M}\vartheta_l s_l\le \tau\alpha && \\
       & \sup_{\xi\in\Xi}\Big\{f(x,\xi)+\tau-\sum_{k=1}^n\lambda_k \|\xi_k-{\xi}_k^l\|^p\Big\}\le s_l  &&  l\in[M]
    \end{aligned}
    \end{align*}
    are met for some $\bm\lambda:=(\lambda_1,\ldots,\lambda_n)\succeq0$ and $s_l\ge 0$, where $l\in[M]$.
\end{prop}

In the next proposition, we consider the case where the function $f$ is piecewise affine with respect to $\xi$. Therefore, we derive a general reformulation when the function is expressed as the maximum of affine functions over a polyhedral set. 

\begin{prop} \longthmtitle{Equivalent characterization of \eqref{eq:DR:CVaR:constraints:1} for piecewise affine constraints} \label{prop:CVaR:piecewise:affine:functions}
    Let Assumption \ref{assumption:P_hat} hold and consider the compact set $\Xi:=\{ \xi\in\mathbb R^d: C\xi\preceq h \}$, where $C\in\Rat{q\times d}$ and $h\in\Rat{q}$. Assume also that $f(x,\xi):=\max_{j\in[m]} \langle x,A_j\xi\rangle+b_j(x)$, where $A_j\in\mathbb R^{\ell\times d}$ and each $b_j:\mathbb R^n\to\Rat{}$ is convex. Then, \eqref{eq:DR:CVaR:constraints:1} consists of all $x\in\mathcal X$ for which the constraints
    \begin{align*}
    \begin{aligned}
        & \langle\boldsymbol{\lambda, \varepsilon}\rangle+\sum_{l=1}^{M} \vartheta_l s_l\le \tau\alpha && \\
& b_j(x)+\tau+\langle A_j^\top x-C^\top \eta_{lj},  {\xi}^l\rangle+\langle{\eta_{lj}}, h\rangle \le s_l && \quad l\in[M],\; j\in[m] \\
& \|\textup{pr}_k^{\bm d}(A_j^\top x- C^\top {\eta_{lj}})\|_* \le \lambda_k && \quad l\in[M],\; j\in[m],\; k\in[n]
\end{aligned}
    \end{align*}  
    are met for some $\bm\lambda:=(\lambda_1,\ldots,\lambda_n)\succeq 0$ and $s_l\ge0$, $\eta_{lj}\succeq0$, $l\in[M]$, $j\in[m]$, with $\bm d:=(d_1,\ldots, d_n)$. 
\end{prop}

Applying Propositions \ref{prop:reformulations:CVaR:convex:functions} and \ref{prop:CVaR:piecewise:affine:functions} to the distributionally robust ${\rm CVaR}$ constraint of \eqref{eq:DRCCP:CVaR}, we obtain the following two corollaries. These respectively yield finite-dimensional reformulations for convex distributionally robust chance-constrained problems and tractable reformulations for problems with piecewise affine constraints. 

\begin{corollary} \longthmtitle{Epigraphical reformulation of CVaR-constrained problem \eqref{eq:DRCCP:CVaR}} \label{cor:DRCCP:epi:reformulations}
    Under Assumption \ref{assumption:convex:concave}, problem \eqref{eq:DRCCP:CVaR} is equivalent to the convex program
    \begin{align*}
    \begin{aligned}
       \inf_{\underset{l\in[M],\; i\in[I]}{x\in\mathcal X,\tau_i,\bm\lambda_i,s_{li}}} \; & \langle g, x \rangle && \\
        {\rm s.t.}\quad\;\;\; & \langle\bm\lambda_i, \bm\varepsilon\rangle+\sum_{l=1}^{M}\vartheta_ls_{li}\le \tau_i\alpha_i && \quad i\in[I],\\
        & \sup_{\xi\in\Xi}\Big\{f_i(x,\xi)+\tau_i-\sum_{k=1}^n\lambda_{ki} \|\xi_k-{\xi}_k^l\|^p\Big\}\le s_{li} && \quad  l\in[M],\; i\in[I] \\
        & s_{li}\ge 0 && \quad  l\in[M],\; i\in[I].
    \end{aligned}
    \end{align*}
    \end{corollary}


\begin{corollary} \longthmtitle{Tractable reformulation  of \eqref{eq:DRCCP:CVaR} with piecewise affine constraints}
    \label{cor:DRCCP:reformulation}
      Assume further that 
      $\Xi:=\{ \xi\in\mathbb R^d: C\xi\preceq h \}$ is compact, where $C\in\Rat{q\times d}$ and $h\in\Rat{q}$, and $f_i(x,\xi):=\max_{j\in[m]} \langle x, A_{ji}\xi\rangle+b_{ji}(x)$, where $A_{ji}\in\mathbb R^{\ell\times d}$ and each $b_{ji}:\mathbb R^\ell\to\Rat{}$ is convex. Then  \eqref{eq:DRCCP:CVaR} is equivalent to the convex optimization problem 
\begin{alignat*}{3}
\inf_{ \underset{l\in[M], j\in[m],\; i\in[I]}{x\in\mathcal X,\tau_i,\bm\lambda_i,s_{li},\eta_{lji}}}\; & \langle g, x \rangle && \\
{\rm s.t.} \quad\quad\; & \langle{\bm\lambda_i, \bm\varepsilon}\rangle+\sum_{l=1}^{M} \vartheta_l s_{li}\le \tau_i\alpha && \quad i\in[I] \\
& b_{ji}(x)+\tau_i+\langle A_{ji}^\top x-C^\top \eta_{lji}, {\xi}^l\rangle+\langle {\eta_{lji}}, h\rangle \le s_{li} && \quad l\in[M],\; j\in[m],\; i\in[I]\\
& \left\|\textup{pr}_k^{\bm d} (A_{ji}^\top x- C^\top {\eta_{lji}})\right\|_* \le \lambda_{ik} && \quad l\in[M], \; j\in[m],\; i\in[I],\; k\in[n] \\
& s_{li}\ge0 \quad \eta_{lji}\succeq0 &&  \quad l\in[M],\; j\in[m],\; i\in[I],
\end{alignat*}  
where $\bm\lambda_i:=(\lambda_{i1},\ldots,\lambda_{in})$ and $\bm d:=(d_1,\ldots, d_n)$.
\end{corollary}

\section{Clustering the product empirical distribution} \label{sec:numerical:complexity}

In the previous sections, we established tractable reformulations for various classes of DRO problems over MTHs that are centered at an atomic reference distribution $\widehat P_\xi$. Recalling our motivation from data-driven problems where the uncertainty $\xi$ consists of several independent components, in this section we turn our attention to the case where $\widehat P_\xi$ has a product atomic structure like the product empirical distribution $\boldsymbol{P}_\xi^N$ in \eqref{eq: prod empiricals} (c.f. Figure \ref{fig:virtual:samples}). The motivation for choosing it as a reference distribution comes from the fact that the corresponding MTH enjoys favorable probabilistic guarantees of containing the true underlying distribution~\cite{LMC-TO-DB:23}.   
Nevertheless, these come with the cost of a higher complexity for the reformulations derived in the previous sections, which may increase exponentially with the number of independent lower-dimensional components $n$. For example, given a random vector with $n=5$ lower dimensional components and $100$ collected samples, we would get a reference distribution comprising $M=N^n=10^{10}$ atoms. Since each atom corresponds to multiple variables and constraints in the associated reformulation, directly using $\boldsymbol{P}_\xi^N$ for large $n$ renders the results derived so far computationally intractable. 

To address this issue, we cluster the product empirical distribution $\bm P_\xi^N$ into an atomic distribution which can yield DRO problems of an acceptable complexity. We also investigate how much it is required to enlarge the vector of transport budgets to retain the statistical guarantees of the new hyperrectangle containing the true distribution. Throughout the section, we will consider random variables $\xi$ supported on a subset $\Xi$ of $\mathbb R^d\equiv \Rat{d_1}\times\ldots\times\Rat{d_n}$ and we fix an arbitrary norm $\|\cdot\|$ on $\mathbb R^d$. 
\begin{figure}[t!]
    \centering
    \includegraphics[width=.9\linewidth]{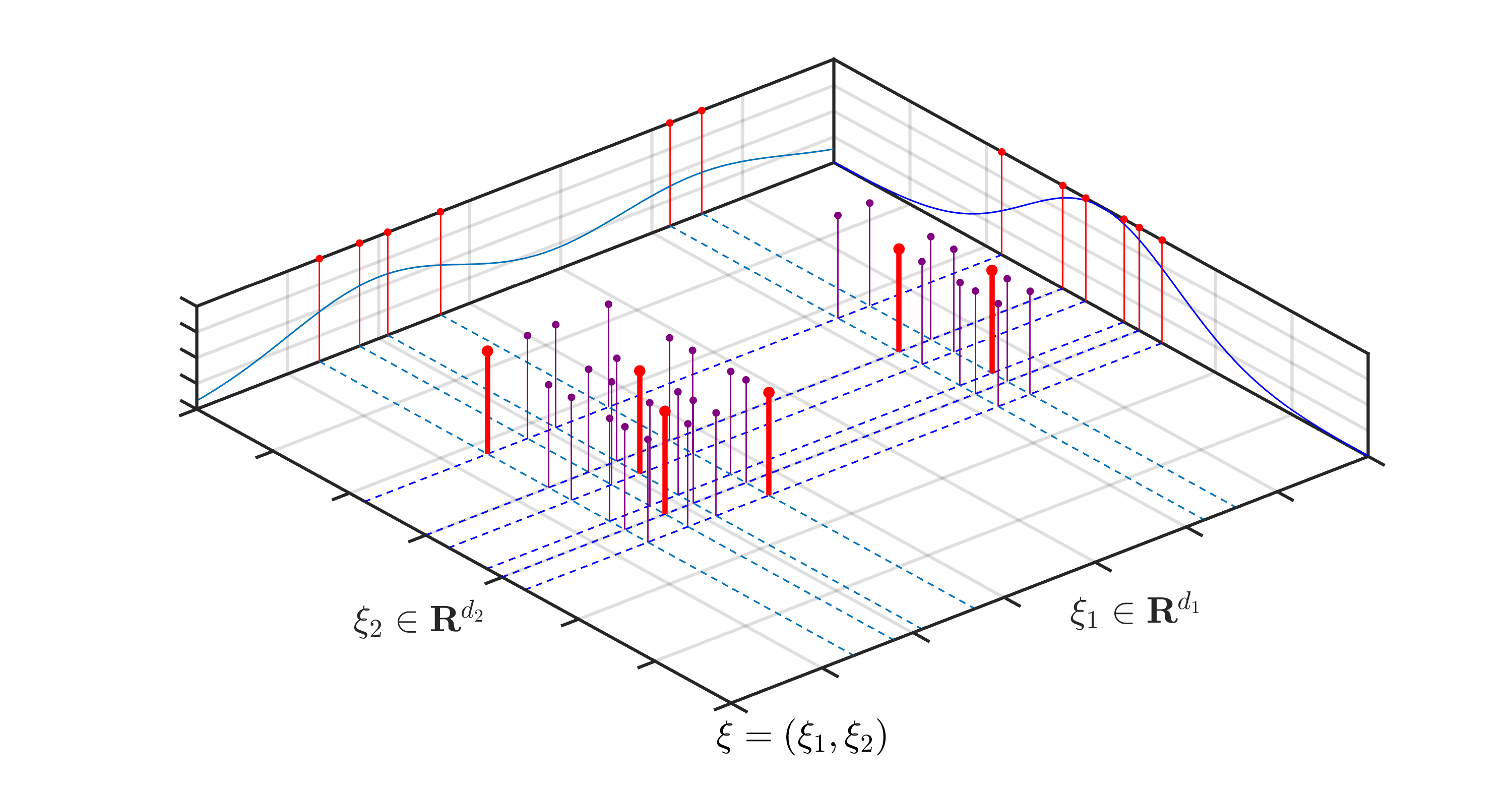}
    \caption{The figure shows the empirical distribution $P_\xi^N$ of  $N=6$ samples, represented by the thick red impulses, and the product empirical distribution $\bm P_\xi^N$, represented by the thin purple impulses. The product empirical distribution is the product of the marginal empirical distributions $P_{\xi_1}^N$ and $P_{\xi_2}^N$ of the independent components of $\xi$, which are depicted by the thin red impulses.}
    \label{fig:virtual:samples}
\end{figure}

Next, we present strategies to construct manageable reference distributions by clustering the atoms of $\bm P_\xi^N$. The clustered distributions enjoy the benefits of having a lower complexity than $\bm P_\xi^N$ while being typically closer to the true distribution compared to empirical models that only rely on i.i.d. samples of $\xi$. Clustering methods to reduce the complexity of data-driven DRO problems have also been considered in \cite{DL-SM:21,FF-PJG:21}. Our key distinctive feature here is to adapt the MTH discrepancy structure to the clustered distributions so that the probabilistic guarantees associated with $\mathcal T_p(\bm P_\xi^N,\bm\eps)$ are maintained for the shifted hyperrectangle.

To achieve this objective, we consider the following approaches:
\begin{itemize}
\item \textbf{Direct clustering:} The first approach is to directly cluster the points on which $\bm P_\xi^N$ is supported into a discrete distribution $\widehat P_\xi$, which has a smaller number of atoms.

\item \textbf{Component-wise clustering:} The second option is to obtain clustered versions $\widehat P_{\xi_k}$ of the marginal empirical distributions $P_{\xi_k}^N$, $k\in[n]$, and deduce a lower-complexity reference model from their product $\widehat P_{\xi_1}\otimes\cdots\otimes\widehat P_{\xi_n}$. 

\item \textbf{Multi-component clustering:} This option provides a middle ground, where one can form $\ell$ groups of consecutive components of $\xi$, 
%
decompose $\bm P_\xi^N$ into products with marginals $\bm P_{(\xi_1,\ldots,\xi_{m_1})}^{N},\ldots, \bm P_{(\xi_{m_{\ell}},\ldots,\xi_{n})}^{N}$, cluster them, and then take the product of these clusters. 
\end{itemize}

The derivation of clusters that are optimal in terms of their Wasserstein discrepancy from the reference distribution $\bm P_\xi^N$ can be performed using Lloyd's algorithm 
\cite[Chapter 9]{CMB-NSN:06},
which is guaranteed to converge to local minima. This is convenient since the complexity of each gradient step of the algorithm is linear in the number of points to be clustered. The clustering process can be further adapted by adjusting the complexity of lower-dimensional clusters and the choice to cluster in a monolithic or a component-wise fashion.

By modifying the reference distribution of the hyperrectangles, the probabilistic guarantees characterizing $\mathcal T_p(\bm P_\xi^N,\bm\eps)$ (cf. \cite[Proposition 5.2]{LMC-TO-DB:23}) do not necessarily hold anymore. Yet, we show how they can be recovered by appropriately adjusting the transport budgets $\eps_k$ in each direction. For simplicity, we only consider the two extreme cases where we either directly cluster all the samples into a monolithic distribution or cluster all marginal empirical distributions and form their associated product.
 \begin{prop} \longthmtitle{Containment guarantees via hyperrectangle inflation} \label{prop:inflation:containment}
 Assume that $P_\xi\in \mathcal T_p(\bm P_\xi^N,\bm\eps)$ and consider the reference distribution $\widehat P_\xi$. 
 
 \noindent (i) If there is a transport plan $\widehat{\pi}\in\mathcal C(\widehat P_\xi,\bm P_\xi^N)$ such that $\int_{\Xi \times \Xi} \|\eta_k-\zeta_k\|^p d\widehat\pi(\eta,\zeta)\le\epsilon_k^p$ holds with some $\epsilon_k\ge0$, for all $k\in[n]$, then 
 \begin{align} \label{inflation:containment}
P_\xi\in \mathcal T_p(\widehat P_\xi,\bm\eps+\bm\epsilon)
 \end{align}
 where $\bm\epsilon:=(\epsilon_1,\ldots,\epsilon_n)$. 

 \noindent (ii) If additionally $\widehat P_\xi$ has the product structure $\widehat P_{\xi_1}\otimes\cdots\otimes\widehat P_{\xi_n}$, then \eqref{inflation:containment} also holds with $\bm\epsilon:=(\epsilon_1,\ldots,\epsilon_n)$ and $\epsilon_k:=W_p(P_{\xi_k}^N,\widehat P_{\xi_k})$, for all $k\in[n]$.
\end{prop}

\begin{proof}
Since $P_\xi\in \mathcal T_p(\bm P_\xi^N,\bm\eps)$, there is by \eqref{transport:hyperrectangle:Tp} a transport plan $\pi$ between $\bm P_\xi^N$ and $P_\xi$ so that   
 \begin{align} 
 \label{component:cost}
\int_{\Xi\times\Xi} \|\zeta_k-\xi_k\|^p {\rm d}\pi(\zeta,\xi)\le\varepsilon_k^p
\end{align}
for all $k\in[n]$. To prove part (i), let  $\widehat\pi$ be a trassport plan as in the statement. From the Gluing Lemma~\cite{CV:08}, there exists a distribution $\widetilde\pi$ on $\Xi^3$ such that $\textup{pr}_{12\#}(\widetilde\pi)=\widehat\pi$ and $\textup{pr}_{23\#}(\widetilde\pi)=\pi$. Then if we consider the transport plan $\pi':=\textup{pr}_{13\#}\widetilde\pi$, it follows that 
\begin{align*}
 \Big(\int_{\Xi \times \Xi} \|\eta_k- \xi_k\|^p d\pi'(\eta,\xi)\Big)^{\frac{1}{p}}   &= \Big(\int_{\Xi^3} \|\eta_k- \xi_k\|^p d\widetilde\pi(\eta,\zeta,\xi)\Big)^{\frac{1}{p}} \\ & \le \Big(\int_{\Xi^3} (\|\eta_k- \zeta_k\| + \| \zeta_k-\xi_k \|)^p d\widetilde\pi(\eta,\zeta,\xi)\Big)^{\frac{1}{p}}   
\\
&\le \Big(\int_{\Xi^3} \|\eta_k- \zeta_k\|^p d\widetilde\pi(\eta,\zeta,\xi)\Big)^{\frac{1}{p}}
+\Big(\int_{\Xi^3} \|\zeta_k-\xi_k\|^p d\widetilde\pi(\eta,\zeta,\xi)\Big)^{\frac{1}{p}} \\ & = \Big(\int_{\Xi\times\Xi} \|\eta_k-\zeta_k\|^p d\widehat\pi(\eta,\zeta)\Big)^{\frac{1}{p}} + \Big(\int_{\Xi \times \Xi} \|\zeta_k-\xi_k\|^p d\pi(\zeta,\xi)\Big)^{\frac{1}{p}} \\ &\le \eps_k+\epsilon_k, 
\end{align*}
for all $k\in[n]$. Here, we used the triangle inequality in the first inequality, Minkowski's inequality in the second one, and Fubini's theorem in the last equality. By the definition of the MTH~\eqref{transport:hyperrectangle:Tp}, this establishes part (i). 
     
To show part (ii), since $W_p(P_{\xi_k}^N,\widehat P_{\xi_k})= \epsilon_k$, there exists by \cite[Theorem 4.1]{CV:08}  an optimal transport plan $\widehat\pi_k$ that minimizes the transport cost between $P_{\xi_k}^N$ and $\widehat P_{\xi_k}$ and therefore  
\begin{align*}
\int_{\Xi \times \Xi} \|\xi_k- \zeta_k\|^p d\widehat\pi_k(\xi_k,\zeta_k) = \epsilon_k^p 
\end{align*}
for each $k\in[n]$. Then if we define 
\begin{align}
\widehat\pi:=T_{\#}\bigotimes_{k=1}^n \widehat\pi_k, 
\end{align}
where $T:\prod_{k=1}^n \Xi_k\times\Xi_k \to \prod_{k=1}^n \Xi_k \times\prod_{k=1}^n \Xi_k $ is the linear map $T(\zeta_1,\xi_1,\ldots,\zeta_n,\xi_n):=(\zeta_1,\ldots,\zeta_n,\xi_1,\ldots,\xi_n)$, one can readily check as in \cite[proof of Proposition 4.4]{LMC-TO-DB:23} that it is a transport plan between $\bm P_\xi^N$ and $\widehat P_\xi$.  Using again the Gluing Lemma between $\widehat\pi$ and $\pi$ in \eqref{component:cost}, it follows exactly as in the proof of part (i) that   
\begin{align*}
\Big(\int_{\Xi \times \Xi} \|\eta_k-\xi_k\|^p d\pi'(\eta,\xi) \Big)^{\frac{1}{p}}  &\le \Big(\int_{\Xi\times\Xi} \|\xi_k- \eta_k\|^p d\pi(\xi,\eta) \Big)^{\frac{1}{p}}+ \Big(\int_{\Xi\times\Xi} \|\eta_k-\zeta_k\|^p d\widehat\pi(\eta,\zeta) \Big)^{\frac{1}{p}} \\ &\le \eps_k+\epsilon_k, 
\end{align*}
where $\pi'$ is again a transport plan between $\widehat P_\xi$ and $P_\xi$. Thus, we deduce from~\eqref{transport:hyperrectangle:Tp} that \eqref{inflation:containment}  also holds with $\bm \epsilon$ as in the statement of part (ii), which concludes the proof.
\end{proof}
This result clarifies how much we need to inflate the hyperrectangle to retain the probabilistic guarantees of containing the data-generating distribution. By clustering the product empirical distribution into $\widehat P_\xi$ using Loyd's algorithm, we also obtain a transport plan that associates each atom of $\bm P_\xi^N$ to its cluster and enables the direct computation of $\epsilon_k$ in case (i). Analogously, in case (ii), the Wasserstein distances $\epsilon_k$ between $P_{\xi_k}^N$ and $\widehat P_{\xi_k}$ 
can be easily computed by solving a linear program~\cite[section 3.1]{GP-MC:19}, or directly deduced through when $\widehat P_{\xi_k}$ is not specified beforehand and is determined using Lloyd's algorithm. Since the motivation to introduce MTHs for data-driven problems is their faster shrinkage compared to Wasserstein balls, in Appendix~\ref{sec:appendix:clustering} we also establish to which degree such desirable shrinkage rates are retained for clustered reference distributions.


\section{Simulation example} \label{sec:simulation:example}
In this section, we provide a real-world example that shows the superiority of MTHs centered at clustered product distribution over monolithic ambiguity balls, validating the theoretical results. In particular, we compare the performance of traditional Wasserstein balls and MTHs, when the latter are centered either at the product empirical distribution or its clustered version.

\subsection{Problem formulation}
We consider a power dispatch problem, where a daily power demand $d+\xi_2$ needs to be covered. The demand consists of a nominal term $d$, and a stochastic fluctuation $\xi_2$ representing the gap between the nominal and actual demand. To meet the demand, we exploit the power $\xi_1$ from a renewable energy resource, which is random, as it depends on the daily weather conditions. Since both $\xi_1$ and $\xi_2$ are random, there is no guarantee that the daily power demand $d+\xi_2$ can be met the renewable resource $\xi_1$. For this reason, we want to ensure that the demand is covered with probability at least $1-\alpha$ by ordering the minimum extra amount $x$ of power from the market (cf. Figure~\ref{fig:sim:example:scheme}).
\begin{figure}
    \centering
    \includegraphics[width=0.85\linewidth]{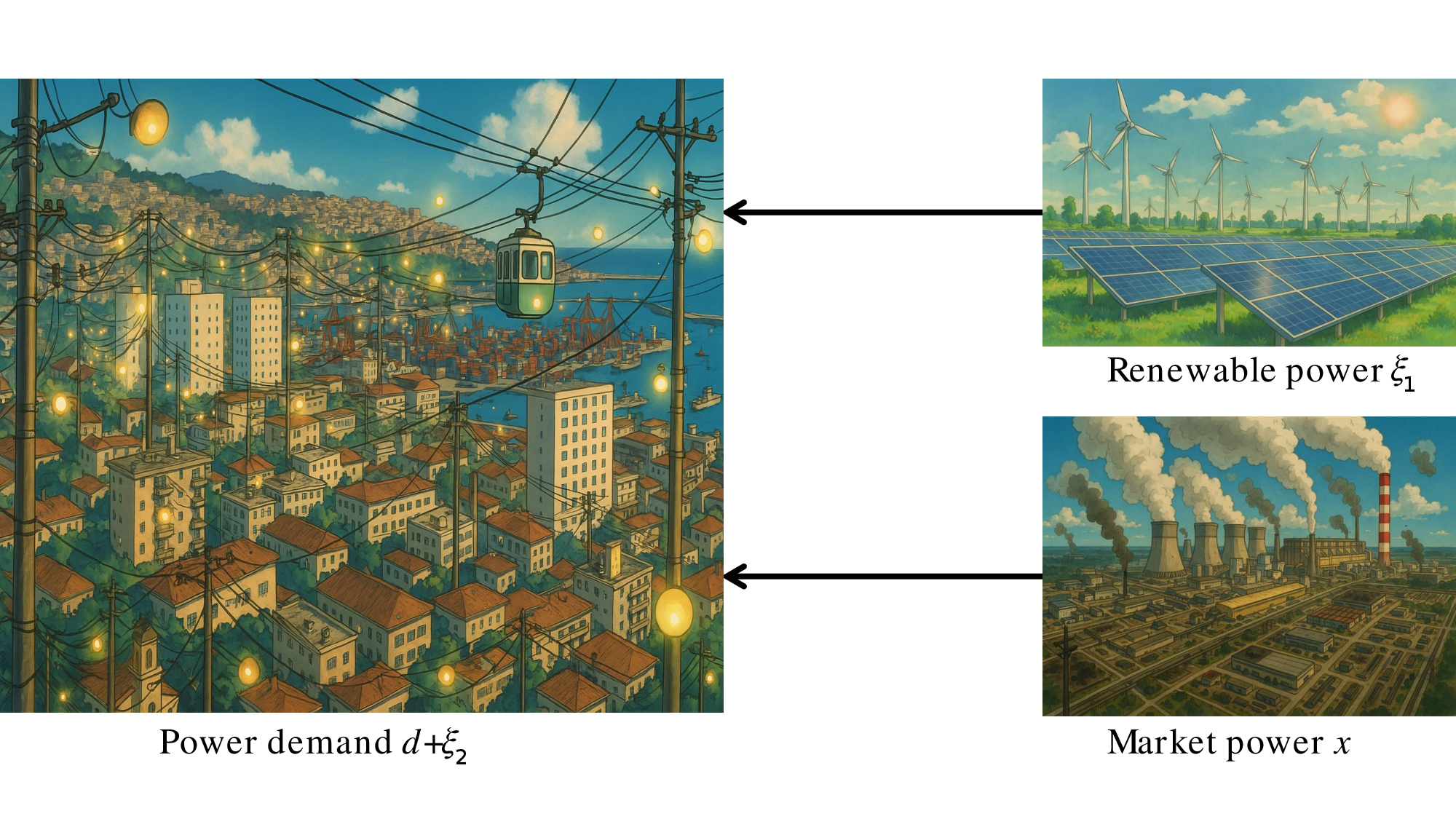}
    \caption{Illustration of the considered power dispatch problem. The goal is to cover the power demand $d+\xi_2$ with probability at least $1-\alpha$, using the renewable power $\xi_1$ and the smaller possible amount $x$ of power from the market.}
    \label{fig:sim:example:scheme}
\end{figure}
This yields the chance-constrained problem 
\begin{equation} \label{eq:sim:1}
    \begin{aligned} 
        \min_{x\ge0} \; & x \\
        {\rm s.t.} \; & \mathrm{CVaR}_{1-\alpha}^\mathbb{P}(d+\xi_2-\xi_1-x)\le 0,
    \end{aligned}
\end{equation}
where the $\mathrm{CVaR}$ constraint ensures the satisfaction of the chance constraint $\mathbb P(d+\xi_2-\xi_1-x\le0)\ge 1-\alpha$ as in Section~\ref{sec:DRCCP}.  We assume that the distributions of $\xi_2$ and $\xi_1$ are unknown and we only have access to $N$ i.i.d. samples $\xi^1,\ldots,\xi^N$ of $\xi:= (\xi_1,\xi_2)$. Thus, instead of \eqref{eq:sim:1} we solve the distributionally robust problem%

\begin{equation} \label{eq:sim:2}
    \begin{aligned} 
        \min_{x\ge0} \; &x \\
        \mathrm{\rm s.t.} \; & \sup_{P\in\mathcal P^N}\mathrm{CVaR}_{1-\alpha}^{P}(d+\xi_2-\xi_1-x)\le0,
    \end{aligned}
\end{equation}
for three distinct data-driven ambiguity sets $\mathcal P^N$. In particular, we consider the cases where $\mathcal P^N$ is the Wasserstein ball $\mathcal B\equiv\mathcal B_1(P_\xi^N, \varepsilon)$ with $P_\xi^N:=\frac{1}{N}\sum_{i=1}^N \delta_{\xi^i}$, the MTH $\mathcal T\equiv\mathcal T_1(\bm P_\xi^N, \boldsymbol{\varepsilon})$ with $\bm P_\xi^N = P_{\xi_1}^N \otimes P_{\xi_2}^N$, and the MTH $\mathcal T_{\rm cl}\equiv \mathcal T_1(\widehat P_\xi, \boldsymbol{\varepsilon})$, where $\widehat P_\xi$ is a product distribution whose marginals are obtained by clustering each marginal of $\bm P_\xi^N$.

Let $\Xi:=\{\xi\in\mathbb R^2:\;C\xi\le h\}$ be the support of $P_\xi$. We also assume $\Rat{2}$ equipped with the $1$-norm and denote $c:=(-1,1)$. Using Corollary \ref{cor:DRCCP:reformulation}, we reformulate \eqref{eq:sim:2} as the linear program
    \begin{align*}
        \min_{\underset{s_l,\eta_l,l\in[M]}{x\ge0,\bm\lambda, \tau}} \; & x \\
        {\rm s.t.}\quad & \langle \bm\lambda,\bm\eps\rangle+\sum_{l=1}^{M} \vartheta_l s_l \le \tau\alpha \\
        & d-x+\tau+\langle c, \xi^l\rangle + \langle \eta_l, h-C\xi^l\rangle \le s_l && l\in[M], \\
        & \big\| \textup{pr}_k^{(1,1)} (c- C^\top \eta_l)\big\|_\infty \le \lambda_k && l\in[M],\; k\in[2],\\
        & s_l\ge0,\;\eta_l\ge0 && l\in[M],
    \end{align*}
where $\bm\lambda:=(\lambda_1,\lambda_2)$. The solution of \eqref{eq:sim:2} is random due to the samples we use to construct the ambiguity set in each case. Depending on the ambiguity set we denote the solution by $\widehat x_{\mathcal B}$, $\widehat x_{\mathcal T}$, and $\widehat x_{\mathcal T_{\rm cl}}$, respectively. 

\subsection{Simulation results}
For the simulations, we choose the distribution $P_\xi:=P_{\xi_1} \otimes P_{\xi_2}$ with 
\begin{equation*}
   \begin{aligned}
       P_{\xi_1} &:= 0.4\mathcal U \big([11, \; 16 ] \big)+ 0.6\mathcal U \big([24, \; 27 ] \big) \\
       P_{\xi_2} &:= 0.6\mathcal U \big([3, \; 6 ] \big)+ 0.4\mathcal U \big([10, \; 11 ] \big),
   \end{aligned}
\end{equation*}
where $\mathcal U$ denotes the uniform distribution on the designated set. We also select the parameters $d=4.5$ and $\alpha=0.2$ and exploit $N=20$ samples to build the empirical and product empirical distributions. We cluster the latter using $M=9\times 8=72$ atoms, by clustering $P_{\xi_1}^N$ and $P_{\xi_2}^N$ into $9$ and $8$ atoms, respectively. Next, we tune the size of the three ambiguity sets so that 
\begin{align*}
\sup_{P\in\mathcal P^N}\mathrm{CVaR}_{1-\alpha}^{P}(d+\xi_2-\xi_1-x)\le0
\end{align*}
with confidence at least $0.9$ for each. 
To this end, we generate a large number of realizations of $N=20$ data samples and construct the respective centers $P_\xi^N$, $\widehat P_\xi$, and $\bm P_\xi^N$. We solve the optimization problem over different $\varepsilon$ values for the radius of the corresponding Wasserstein ball and the smallest ball enclosing the MTHs, and determine the constraint satisfaction frequencies of the associated decisions for the true distribution (cf. Figure \ref{fig:beta:wrt:epsilon}).
\begin{figure}[t]
    \centering
    \includegraphics[width=0.95\linewidth]{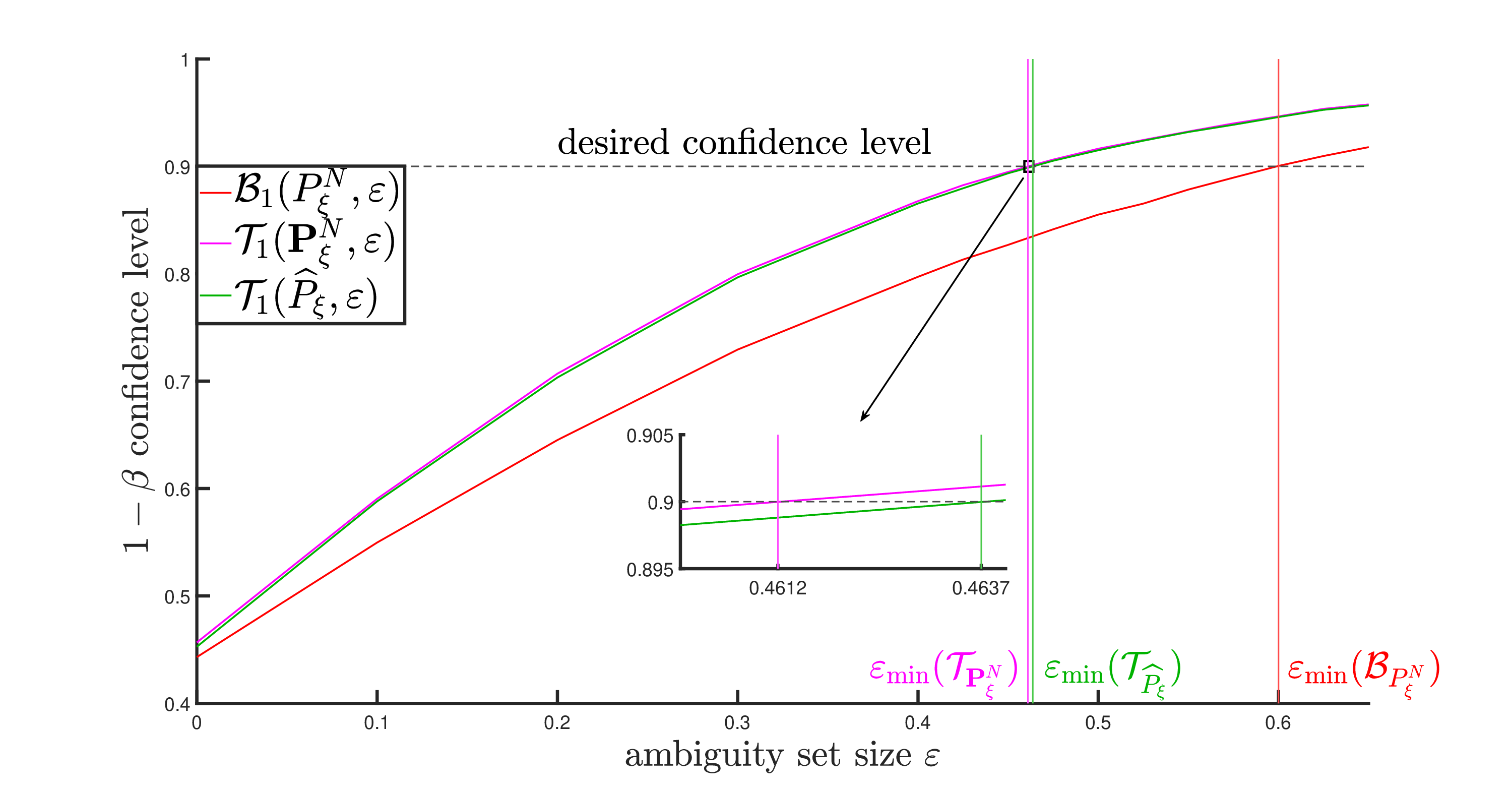}
    \caption{Satisfaction probability of the ${\rm CVaR}$ constraint for each ambiguity set with respect to its radius $\varepsilon$ for the parameters $d=4.5$, $\alpha = 0.2$, $N=20$ and $M=72$. For each hyperrectangle, $\varepsilon$ represents the radius of the smallest ball enclosing it. The minimum radii ensuring the desired 0.9 confidence level for each ambiguity set are indicated by the vertical lines. For every $\eps$, MTHs satisfy the chance-constraint with confidence levels that are always greater than that of the corresponding ambiguity ball, which turns out to be much more conservative. We also observe that despite the clear complexity reduction of clustered distribution compared to the product empirical distribution, the confidence levels of their associated MTHs are very close for every $\eps$.}
    \label{fig:beta:wrt:epsilon}
\end{figure}
 Next, we determine for each ambiguity set the smallest value of $\eps$ for which the ${\rm CVaR}$ constraint is met with the desired 0.9 confidence margin and solve all three instances of the problem for a large number of data realizations. Figure \ref{fig:cumulative:solution} shows the cumulative distribution of these solutions.
\begin{figure}
    \centering
    \includegraphics[width=0.95\linewidth]{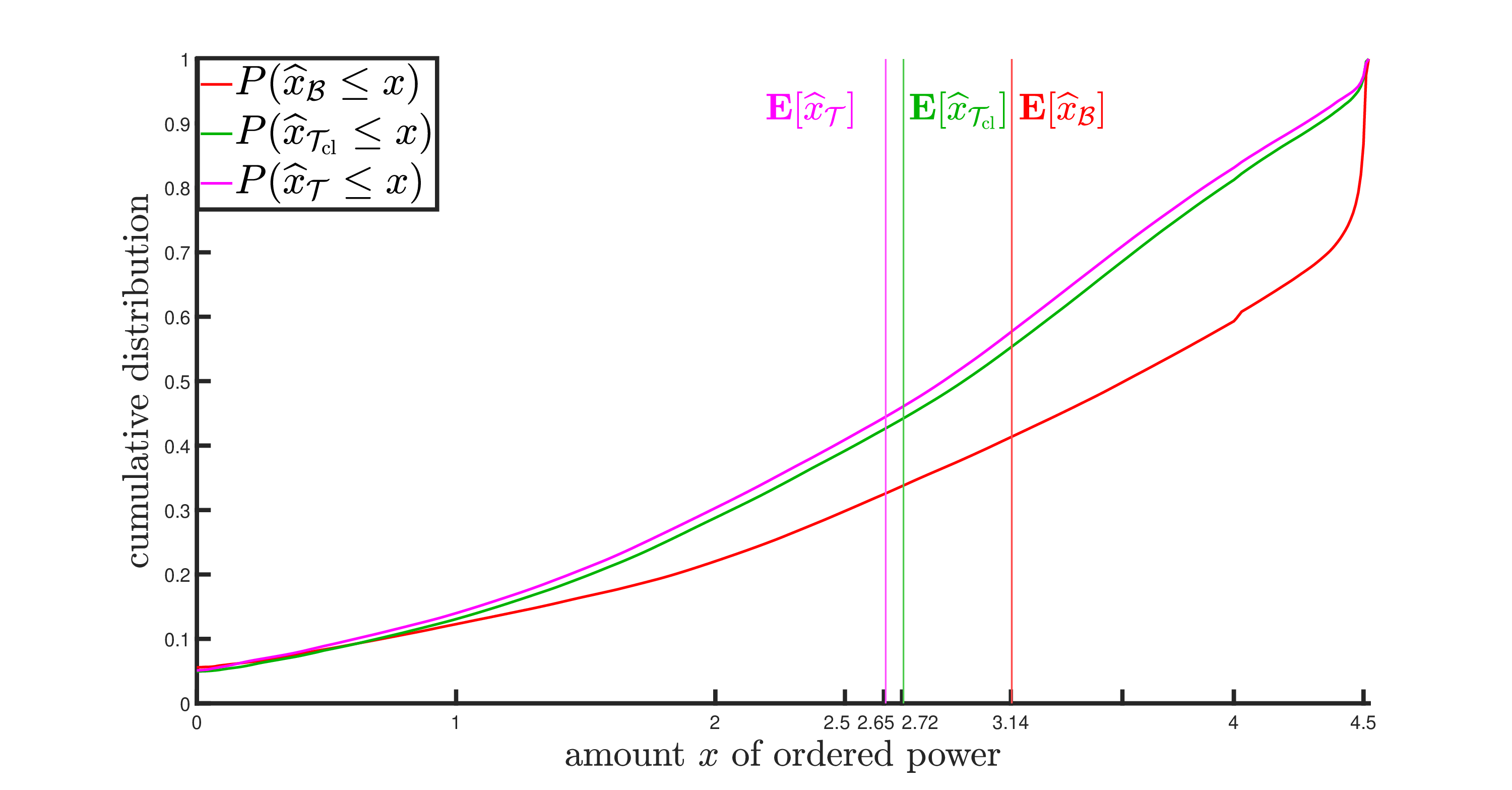}
    \caption{Cumulative distributions of $\widehat x_{\mathcal B}$ in red, $\widehat x_{\mathcal T_{\rm cl}}$ in green and $\widehat x_{\mathcal T}$ in magenta for the smallest size of each ambiguity set that ensures the satisfaction of the ${\rm CVaR }$ constraint with a confidence level of 0.9 for the parameters $d=4.5$, $\alpha=0.2$, $N=20$ and $M=72$. 
    The mean values of the distributionally robust solutions $\widehat x_{\mathcal B}$ in red, $\widehat x_{\mathcal T_{\rm cl}}$ in green, and $\widehat x_{\mathcal T}$ in magenta are depicted with vertical lines. 
    The obtained DRO values with the MTHs are significantly smaller than the ones with the traditional ambiguity balls. Further, the average solution when the MTH is centered at the clustered distribution is very close to the one where the MTH is centered at the product empirical distribution, despite the considerable difference in the number of atoms of the two centers.}
    \label{fig:cumulative:solution}
\end{figure}
%


From Figure \ref{fig:beta:wrt:epsilon}, it is clear that for each $\eps$, the decisions of the rectangles ensure constraint satisfaction for a significantly higher confidence level than the monolithic ball. In particular, to achieve the desired level of $0.9$ confidence, the monolithic ambiguity ball needs to have a radius $\eps_{\rm min}(\mathcal B)=0.6$ that is much larger than the radii $\eps_{\rm min}(\mathcal T)=0.4612$ and $\eps_{\rm min}(\mathcal T_{\rm cl})=0.4637$ of the balls enclosing the two rectangles. We also note that $\mathcal T_{\rm cl}$ has almost similar confidence levels as $\mathcal T$ despite the smaller complexity of its clustered center ($72$ instead of $400$ atoms). It only needs to be slightly larger than $\mathcal T$ to achieve the same guarantees.  

In Figure \ref{fig:cumulative:solution}, we plot the cumulative distributions of the solutions $\widehat x_{\mathcal B}$, $\widehat x_{\mathcal T_{\rm cl}}$ and $\widehat x_{\mathcal T}$, and their expectations using colored vertical lines. 
First, we notice that the distribution mass accumulates faster for the solutions resulting from DRO problems associated with MTHs compared to the solution resulting from the DRO problem associated with the monolithic ambiguity ball. This means that the DRO problems associated with MTHs have solutions that are significantly less conservative while still guaranteeing the satisfaction of the ${\rm CVaR}$ constraint with the desired confidence level.

 \section{Conclusion}  \label{sec:conclusion}

In this paper, we derived tractable reformulations for several classes of distributionally robust stochastic optimization problems where distributional uncertainty is captured through ambiguity sets formed by multiple optimal transport constraints. The sets, called multi-transport hyperrectangles (MTHs) share probabilistic guarantees of containing the unknown distribution, which scales conveniently with the number of samples when the uncertainty consists of independent low-dimensional components. For the reformulations, we also established basic properties of MTHs, including weak compactness and conditions ensuring finiteness of their associated DRO values. 

We derived reformulations for objective functions whose dependence on the uncertainty is quadratic or expressed as the maximum of finitely many concave functions, including piecewise affine ones. Using the results for the piecewise affine class, we solved uncertainty quantification problems over unions of polyhedral sets. 
We also derived tractable reformulations for distributionally robust chance-constrained problems over MTHs. To mitigate the complexity of the reformulations for data-driven problems, we clustered the reference distribution of the MTHs, obtaining tradeoffs between conservativeness and complexity reduction. 
The numerical simulations illustrate the benefits of MTHs compared to traditional Wasserstein balls when the uncertainty consists of independent components. Clustering the center of the MTH yielded results of comparable performance under a noticeable complexity reduction. 

Future work will exploit the tools developed in this paper to solve distributionally robust control problems with independent uncertainty components. These include stochastic model predictive control algorithms where uncertainty is independent across stages and the control of multi-agent systems where agents are subject to independent disturbances.

\appendix

\section{Technical proofs}
\label{sec:appendix:A}

\subsection{Proofs from Section \ref{sec:basic:results}}


In this part, we provide the proofs of some statements of Section \ref{sec:basic:results}.
To prove Theorem~\ref{thm:weak:compactness}, we also introduce some further preliminaries and results from probability theory. A collection $\mathcal S\subset \mathcal P(\Xi)$ of probability distributions is called tight if, for any $\varepsilon>0$, there exists a compact subset $B\subset \Xi$ such that $P(\Xi \backslash B)\le\varepsilon$, for all $P\in\mathcal S$. The following theorem of Prokhorov provides a link between weak compactness and tightness.
\begin{thm}
\longthmtitle{Prokhorov's theorem} A collection $\mathcal S \subset \mathcal P(\Xi)$ of probability measures is tight if and only if the weak closure of $\mathcal S$ is weakly compact in $\mathcal P(\Xi)$.
\label{thm:Prokhorov}
\end{thm}

We will exploit the following result from~\cite{CV:08} that establishes the lower semicontinuity of the transport-cost functional. We adopt it here for positive transport costs.
\begin{lemma}
    \longthmtitle{Lower semicontinuity of the cost functional \cite[Lemma~4.3]{CV:08}} 
    Let $\Xi_1$ and $\Xi_2$ be two Polish spaces and $c:\Xi_1\times\Xi_2\xrightarrow[]{}[0,+\infty]$ be a lower semicontinuous cost function. 
    Then $F:\pi\xrightarrow[]{}\int_{\Xi_1\times\Xi_2} c d\pi$ is lower semicontinuous on $\mathcal P(\Xi_1\times\Xi_2)$, equipped with the topology of weak convergence.
    \label{lemma:cost:functional:lower:semi:continuity}
\end{lemma}
We also need the next result, which certifies the tightness of the set of transport plans between tight sets of probability distributions.
\begin{lemma}
    \longthmtitle{Tightness of couplings \cite[Lemma~4.4]{CV:08}} 
    Let $\Xi_1$ and $\Xi_2$ be two Polish spaces and let $\mathcal R_1$ and $\mathcal R_2$ be tight subsets of $\mathcal P(\Xi_1)$ and $\mathcal P(\Xi_2)$, respectively. Then the set $\mathcal C(\mathcal R_1, \mathcal R_2)$ of all couplings with marginals in $\mathcal R_1$ and $\mathcal R_2$, respectively, is itself tight in $\mathcal P(\Xi_1\times\Xi_2)$.
    \label{lemma:tightness:of:couplings}
\end{lemma}
The next result ensures that properness is maintained on the product of proper spaces under equivalence of any product metric with the considered metric on the product space. 

\begin{lemma} 
\label{lem:product:space:proper}
\longthmtitle{Product of proper metric spaces}  
    Consider the proper spaces $(\Xi_k, \rho_k)$, $k\in[n]$, and let Assumption \ref{assumption:equivalent:metric} hold. Then, the metric space $(\Xi,\rho)$ is also proper. 
\end{lemma}

\begin{proof}
It suffices to show that the ball $B_\rho(\zeta,r):=\{\xi\in\Xi:\;\rho(\zeta,\xi)\le r\}$ is compact for each $\zeta\in\Xi$ and $r>0$. To this end, we exploit Assumption~\ref{assumption:equivalent:metric}(ii), which asserts that there exists $C>0$ such that $\rho_k(\zeta_k,\xi_k)\le C\rho(\zeta,\xi)$ for all $\zeta,\xi\in\Xi$ and $k\in[n]$. This in turn  implies that  
\begin{align*}
         B_\rho(\zeta,r)\subset B_{\rho_1}(\zeta_1,Cr)\times\ldots\times B_{\rho_n}(\zeta_n,Cr).
    \end{align*}
    Since the sets $B_{\rho_k}(\zeta_k,Cr)$, $k\in[n]$ are compact by properness of each $(\Xi_k,\rho_k)$, their cartesian product is also compact, which establishes the result. 
\end{proof}


\begin{prop} \longthmtitle{Weak compactness of the set of transport plans} \label{prop:weak:compactness:Pi}
Let Assumption~\ref{assumption:equivalent:metric} hold and consider the set of transport plans
\begin{align*}
  \Pi_p(Q,\bm\eps):=\big\{ \pi\in\mathcal P(\Xi\times\Xi):\;&\textup{pr}_{1\#}\pi=Q\;\textup{and}\\ &\int_{\Xi\times\Xi}\rho_k(\zeta_k,\xi_k)^p d\pi(\zeta,\xi)\le\eps_k\;\textup{for all}\; k\in[n]\Big\}.
\end{align*}
where $\bm\eps:=(\eps_1,\ldots,\eps_n)\succeq 0$ and $Q$ has finite $p$th moment. Then $\Pi_p(Q,\bm\eps)$ is weakly compact. 
\end{prop}

\begin{proof}
From \cite[Proposition 4.6]{LMC-TO-DB:23} and Assumption \ref{assumption:equivalent:metric}(ii) we can enclose the MTH $\mathcal T_p(Q,\bm\varepsilon)$ inside a Wassertein ball $\mathcal B_p(Q,\varepsilon)$ of a sufficiently large radius $\varepsilon$. Since $(\Xi,\rho)$ is proper by Lemma \ref{lem:product:space:proper} and the reference distribution $Q$ has a finite $p$th moment, we get by \cite[Theorem 1]{MY-DK-WW:21} that this enclosing Wasserstein ball is also weakly compact. 
From the reverse implication of Prokhorov's theorem (Theorem \ref{thm:Prokhorov}) this ball is tight, which implies that $\mathcal T_p(Q,\bm\varepsilon)$ is also tight as it is contained in $\mathcal B_p(Q,\varepsilon)$. Recalling that $\Pi_p(Q,\bm\eps)$ is the set of transport plans between $\{Q\}$ and $\mathcal T_p(Q,\bm\eps)$, by virtue of Lemma \ref{lemma:tightness:of:couplings}, $\Pi_p(Q,\bm\eps)$ is also tight. Thus, to show that $\Pi_p(Q,\bm\eps)$ is weakly compact, it suffices from the direct implication of Prokhorov's theorem to show that it is weakly closed. 

 Since each $\rho_k$, $k\in[n]$ is a metric, its is nonnegative. Thus, we get from Lemma~\ref{lemma:cost:functional:lower:semi:continuity} with $\Xi_1\equiv\Xi_2\equiv\Xi$ that the sublevel sets $F_k^{-1}((-\infty,\eps_k])$ of the functionals $F_k:\pi\xrightarrow[]{}\int_{\Xi\times\Xi} \rho_k^p d\pi$, $k\in[n]$ are weakly closed. Since each of those is equal to the set 
 \begin{align*}
 \Pi_{p,k}(Q,\eps_k):=\Big\{\pi\in \mathcal P(\Xi\times\Xi):\;\textup{pr}_{1\#}\pi=Q\;\text{and}\; \int_{\Xi\times\Xi}\rho_k(\zeta_k,\xi_k)^p d\pi(\zeta,\xi)\le \eps_k\Big\},
 \end{align*}
and taking into account that $\Pi_p(Q,\bm\eps)=\cap_{k=1}^n \Pi_{p,k}(Q,\eps_k)$, it follows that $\Pi_p(Q,\bm\eps)$ is weakly closed as well, which establishes the result.
\end{proof}

Now we have gathered all the necessary ingredients to prove weak compactness of $\mathcal T_p(Q,\bm\eps)$.

\begin{proof}[Proof of Theorem~\ref{thm:weak:compactness}]
    
     First, recall that from \cite[Proposition 4.6]{LMC-TO-DB:23} and Assumption \ref{assumption:equivalent:metric}(ii) we can enclose the MTH $\mathcal T_p(Q,\bm\varepsilon)$ inside a Wassertein ball $\mathcal B_p(Q,\varepsilon)$ of a sufficiently large radius $\varepsilon$. Since $(\Xi,\rho)$ is proper by Lemma \ref{lem:product:space:proper} and the reference distribution $Q$ has a finite $p$th moment, we get by \cite[Theorem 1]{MY-DK-WW:21} that this enclosing ball is weakly compact, which implies tightness of $\mathcal T_p(Q,\bm\eps)$. Hence, in order to show the weak compactness of $\mathcal T_p(Q,\bm\eps)$, it is enough to show that $\mathcal T_p(Q,\bm\eps)$ is closed. To this end let $\{P_j\}_{j\in\mathbb N}\subset\mathcal T_p(Q,\bm\eps)$ be a sequence converging to some $P\in\P(\Xi)$. By the definition of $\mathcal T_p(Q,\bm\eps)$, this sequence consists of the second marginals of a sequence  $\{\pi_j\}_{j\in\mathbb N}\subset\Pi_p(Q,\bm\eps)$, i.e., $P_j=\textup{pr}_{2\#}\pi_j$ for each $j$. Since $\Pi_p(Q,\bm\eps)$ is weakly compact by Proposition \ref{prop:weak:compactness:Pi}, we can extract a subsequence of $\{\pi_j\}_{j\in\mathbb N}$ that converges weakly to some $\pi\in \Pi_p(Q,\bm\eps)$ and denote $P':=\textup{pr}_{2\#}\pi$. With a slight abuse of notation, we also index this sequence by $j$. Therefore, for any continuous and bounded function $g$ on $\Xi$, we have 
    \begin{align*}  \lim_{j\to\infty}\int_{\Xi\times\Xi}g(\xi)dP_j(\xi) & =\int_{\Xi\times\Xi}g(\xi)\textup{pr}_{2\#}\pi_j(\zeta,\xi)=\lim_{j\to\infty}\int_{\Xi\times\Xi}g\circ \textup{pr}_2(\zeta,\xi) \pi_j(\zeta,\xi)\\ 
    & =\int_{\Xi\times\Xi}g\circ \textup{pr}_2(\zeta,\xi) \pi(\zeta,\xi)= \int_{\Xi\times\Xi}g(\xi)\textup{pr}_{2\#}\pi(\zeta,\xi)=\int_{\Xi\times\Xi} g(\xi)dP'(\xi). 
    \end{align*}
 By uniqueness of the weak limit, $P=P'\equiv{\rm pr_{2\#}}\pi$, and since $\pi\in\Pi_p(Q,\bm\eps)$, it follows by the definition of $\mathcal T_p(Q,\bm\eps)$ that $P$ belongs to $\mathcal T_p(Q,\bm\eps)$, which establishes that the set is closed and concludes the proof.
\end{proof}

\begin{proof}[Proof of Theorem \ref{thm:existence:inner:maximization}] 

The proof of part (i) is analogous to the proof of \cite[Theorem 3]{MY-DK-WW:21}. From Assumption~\ref{assumption:equivalent:metric}(ii) and \cite[Proposition 4.6]{LMC-TO-DB:23} we can enclose $\mathcal T_p(Q, \bm\eps)$ inside a sufficiently large Wasserstein  ball. Thus, it follows from \cite[Lemma 1]{MY-DK-WW:21} that the MTH $\mathcal T_p(Q, \bm\eps)$ has a uniformly bounded $p$th moment.
Together with our assumptions on $h$, this implies along the lines of the proof of \cite[Theorem 3]{MY-DK-WW:21} that the map $P\to \bE_P[h]$ is upper semicontinuous on $\mathcal T_p(Q, \bm\eps)$. Combining this with weak compactness of $\mathcal T_p(Q, \bm\eps)$ we establish that $h$ attains its supremum.

To prove part (ii), let $C>0$ and $\zeta\in\Xi$ such that \eqref{eq:h:growth:bound} holds for both $h$ and $-h$, consider the ball $B_\rho(\zeta,R):=\left\{ \xi\in\Xi:\; \rho(\zeta,\xi)\le R \right\}$ for some $R>0$, and let $\{P_k\}_{k\in\mathbb N}$ be a sequence in $\mathcal T_p(Q,\bm\eps)$  converging weakly to some $P_\star$. To prove the weak continuity of $\Psi$, we show that $\lim_{k\to\infty} \Psi(P_k)=\Psi(P_\star)$. Namely, we will show that for any $\epsilon>0$ there exists $k_0\in\mathbb N$ such that 
    \begin{align} \label{Psi:epsilon:increment}
        \left| \Psi(P_k)-\Psi(P_\star)\right|<\epsilon \quad\textup{for all}\; k\ge k_0.
    \end{align}
    To this end, we introduce a truncated version of $h$. We 
    fix $R>0$ and define  
    %
    %
    \begin{align*}
        h_R(\xi):=\min\{h^\star(\xi), C(1+R^r)\}.
    \end{align*}
    where $h^\star(\xi):=\max\{h(\xi),-C(1+R^r)\}$.
    We then have that\footnote{If $\rho(\zeta,\xi)< R$, then $\left|h(\xi) \right|\le C(1+R^r)$ and thus $h(\xi)-h_R(\xi)=0$. If $\rho(\zeta,\xi)\ge R$, then either $h(\xi)\ge C(1+R^r)$, which implies $0\le h(\xi)-h_R(\xi)\le C\rho(\zeta,\xi)^r$ by \eqref{eq:h:growth:bound}, or $h(\xi)\le-C(1+R^r)$, which again by~\eqref{eq:h:growth:bound} implies $-C\rho(\zeta,\xi)^r\le h(\xi)-h_R(\xi)\le0$, or $|h(\xi)|\le C(1+R^r)$, which implies $h(\xi)-h_R(\xi)=0$.} 
    \begin{align*}
        | h(\xi)-h_R(\xi)| \le  \begin{cases} 
        C \rho(\zeta,\xi)^r,  &\text{if } \rho(\zeta,\xi)\ge R \\
        0,  &\text{otherwise}.
        \end{cases} 
    \end{align*}
    This bound implies that there exists $b>0$  such that 
    \begin{align}
        \Big| \int_\Xi h(\xi) dP(\xi) - \int_\Xi h_R(\xi) dP(\xi) \Big| & \le \int_\Xi \left| h(\xi)-h_R(\xi)\right| dP(\xi)  \nonumber \\
        &\le C \int_{\Xi\setminus B_R(\zeta)} \rho(\zeta,\xi)^r dP(\xi) \le C b / R^{(p-r)} \label{eq:int:diff:bound}
    \end{align}
     for all $P\in\mathcal T_p(Q, \bm\eps)$. Indeed, in analogy to the proof of \cite[Theorem 3]{MY-DK-WW:21}, this follows from the fact that $\rho(\zeta,\xi)^r=\rho(\zeta,\xi)^p/\rho(\zeta,\xi)^{r-p}\le \rho(\zeta,\xi)^p/R^{r-p}$  over the domain of integration and \cite[Lemma 1]{MY-DK-WW:21}\footnote{This lemma can be applied to $\mathcal T_p(Q, \bm\eps)$ since it is possible to enclose it inside a sufficiently large Wasserstein ball that is centered at $Q$ \cite[Proposition 4.6]{LMC-TO-DB:23}.}.  
    Then, we obtain from the triangle inequality that  
    \begin{align*}
         | \Psi(P_k)-\Psi(P_\star)| & = \Big| \int_\Xi h(\xi) dP_k(\xi) - \int_\Xi h(\xi) dP_\star(\xi) \Big|  \\
         & \le \Big| \int_\Xi h_R(\xi) d(P_k - P_\star)(\xi) \Big| + \Big| \int_\Xi (h(\xi)- h_R(\xi))dP_k(\xi)\Big|  \\ & \hspace{12em} 
         + \Big| \int_\Xi (h(\xi)- h_R(\xi))dP_\star(\xi)\Big| 
    \end{align*}
    for all $k$. The last two terms can be rendered arbitrarily small by  \eqref{eq:int:diff:bound} for a large enough $R$. Since $h$ is continuous, which implies that $h_R$ is as well, and $P_k$ converges weakly to $P_\star$, we can also make the first term as small as desired for large enough $k$. This establishes \eqref{Psi:epsilon:increment} and concludes the proof.
\end{proof}

\subsection{Proofs from Section \ref{sec:DRO:reformulations}}
Here we gather all proofs from Section \ref{sec:DRO:reformulations}.
\begin{proof}[Proof of Proposition \ref{prop:finite:convex:program}] 
Since $h$ is real-valued by Assumption~\ref{assumption:hj:concave}, it is also bounded on the atoms of $Q$. Thus, it satisfies the conditions of Theorem \ref{prop:dual}, which yields  

\begin{align*}
    \sup_{P_\xi\in\mathcal T_1(Q,\bm\varepsilon)}
     \mathbb{E}_{P_\xi}[h(\xi)] & =\inf_{\bm\lambda\succeq0}\Big\{\langle\bm\lambda, \bm\varepsilon\rangle+\sum_{l\in[M]} \vartheta_l\sup_{\xi\in\Xi}\Big\{h(\xi)-\sum_{k=1}^n\|\xi_{k}^{l}-\xi_k\|\Big\}\Big\}\\
    &  =\inf_{\bm\lambda\succeq0}\Big\{\langle\bm\lambda, \bm\varepsilon\rangle+\sum_{l\in[M]} \vartheta_l\max_{j\in[m]}\sup_{\xi\in\Xi}\Big\{h_j(\xi)-\sum_{k=1}^n\lambda_k \|\xi_{k}^{l}-\xi_k\|\Big\}\Big\},
   \end{align*}
where we used the fact that $h(\xi)=\max_{j\in[m]}h_j(\xi)$ in the last equality. Introducing epigraphical variables $s_l$, $l\in[M]$, the problem is equivalent to 
\begin{align}
\begin{aligned}
     \inf_{\bm\lambda, s_l}\; &\langle\boldsymbol{\lambda, \varepsilon}\rangle+\sum_{l=1}^{M}\vartheta_l s_l && \\
      {\rm s.t.}\; & \sup_{\xi\in\Xi}\Big\{h_j(\xi)-\sum_{k=1}^n\lambda_k \|{\xi}_{k}^{l}-\xi_k\|\Big\}\le s_l && l\in[M], \;j\in[m] \\
      & \lambda_k\ge 0 &&  k\in[n].
   \end{aligned}
\label{convex:partial:reformulation:1}  
\end{align}
%
%
By adopting the arguments from \cite[cf. (12d) \& (12e)]{PME-DK:17} to the vectorial with respect to $\bm\lambda$ form of the constraints in \eqref{convex:partial:reformulation:1}, and considering for each $l\in[M]$ and $j\in[m]$ the dual variable $z_{lj}\in\mathbb R^d$, we get that
\begin{align*}
\sup_{\xi\in\Xi}\Big\{h_j(\xi) & -\sum_{k=1}^n \max_{\| \textup{pr}_k^{\bm d}(z_{lj})\|_*\le\lambda_k}\langle \textup{pr}_k^{\bm d}(z_{lj}),\xi_{k}-{\xi}_k^l\rangle\Big\} \\ = &\sup_{\xi\in\Xi}\;\min_{\| \textup{pr}_k^{\bm d} (z_{lj})\|_*\le\lambda_k,k\in[n]}\Big\{h_j(\xi)-\sum_{k=1}^n \langle \textup{pr}_k^{\bm d}(z_{lj}),\xi_{k}-{\xi}_k^l\rangle\Big\} \nonumber \\
 = & \min_{\| \textup{pr}_k^{\bm d}(z_{lj})\|_*\le\lambda_k,k\in[n]}\;\sup_{\xi\in\Xi}\Big\{h_j(\xi)-\sum_{k=1}^n \langle \textup{pr}_k^{\bm d}(z_{lj}),\xi_{k}-{\xi}_k^l\rangle\Big\}. 
\end{align*}
The last equality follows by compactness of $\{z_{lj}:\|\textup{pr}_k^{\bm d}(z_{lj})\|_*\le\lambda_k, k\in[n]\}$ and the minimax theorem \cite[Proposition 5.5.4]{DPB:09}.
Then each of the first set of constraints in \eqref{convex:partial:reformulation:1} is equivalent to
\begin{align}
\left\{\begin{aligned}
&\sup_{\xi\in\Xi}\{h_j(\xi)-\langle z_{lj}, \xi-{\xi}^l\rangle\}\le s_l  \\
&\| \textup{pr}_k^{\bm d}(z_{lj})\|_*\le\lambda_k, \quad k\in[n], 
\end{aligned}\right. \iff 
\left\{\begin{aligned} 
&\sup_{\xi\in\Xi}\{h_j(\xi)+\langle z_{lj}, \xi\rangle\}-\langle z_{lj},{\xi}^l\rangle\le s_l  \\
& \| \textup{pr}_k^{\bm d}(z_{lj})\|_*\le\lambda_k, \quad k\in[n]. 
\end{aligned}\right. \label{convex:partial:reformulation:3}  
\end{align}
where we substituted $z_{lj}$ by $-z_{lj}$ in the right-hand side of the equivalence. In addition, from Assumption \ref{assumption:hj:concave}, which establishes that $-h_j(\xi)$ is proper, convex, and lower semi-continuous, we obtain by using the exact same arguments as in \cite{PME-DK:17} that
\begin{align}
\sup_{\xi\in\Xi}\{h_j(\xi)+\langle z_{lj}, \xi\rangle\}={\rm cl}\big(\inf_{\upsilon_{lj}}\{\sigma_\Xi(\upsilon_{lj})+[-h_j]^*(z_{lj}-\upsilon_{lj})\}\big),  \label{convex:partial:reformulation:4}
\end{align}
where ${\rm cl}(f)(x):=\liminf_{z\to x}f(z)$ (cf. \cite[Page 14]{RTR-RJBW:10}). Therefore, by taking further into account \eqref{convex:partial:reformulation:1}--\eqref{convex:partial:reformulation:4}, we obtain the reformulation in the statement of the proposition.
\end{proof}

\begin{proof}[Proof of Proposition \ref{prop:Paffine:reformulation}] 
    The proof relies on Proposition \ref{prop:finite:convex:program} and the exact same reasoning as in the proof of \cite[Corollary 5.1]{PME-DK:17}. 
\end{proof}

\begin{proof} [Proof of Proposition \ref{prop:quad:reformulation}]
    By Theorem \ref{prop:dual} and Assumption \ref{assumption:P_hat}, \eqref{eq:inner:maximization} is equal to the value of the dual problem
    \begin{align*}
         \inf_{\bm\lambda\succeq0} \langle \bm\lambda,\bm\epsilon\rangle & + \sum_{l=1}^{M} \vartheta_l\sup_{\xi\in\Xi}\Big\{\xi^\top \mathcal Q\xi+2q^\top \xi-\sum_{k=1}^n \lambda_k \|\xi_k^l-\xi_k\|^2\Big\} \nonumber \\
         & = \inf_{\bm\lambda\succeq0} \langle \bm\lambda,\bm\epsilon\rangle + \sum_{l=1}^{M} \vartheta_l\sup_{\xi\in\Xi}\{\xi^\top \mathcal Q\xi+2q^\top \xi-(\xi^l-\xi)^\top {\rm diag}^{\bm d}(\bm\lambda)(\xi^l-\xi)\} \nonumber \\
         & = \inf_{\bm\lambda\succeq0} \langle \bm\lambda,\bm\epsilon\rangle + \sum_{l=1}^{M} \vartheta_l\sup_{\xi\in\Xi}\{\xi^\top (\mathcal Q- {\rm diag}^{\bm d}(\bm\lambda))\xi+2(q+ {\rm diag}^{\bm d}(\bm\lambda)\xi^l)^\top \xi-{\xi^l}^\top {\rm diag}^{\bm d}(\bm\lambda)\xi^l\}. 
    \end{align*}
    Since $\Xi\equiv\Rat{d}$, we have that 
    \begin{align*}
        \sup_{\xi\in\Rat{d}}\{\xi^\top (\mathcal Q- {\rm diag}^{\bm d}(\bm\lambda))\xi+2(q+ {\rm diag}^{\bm d}(\bm\lambda)\xi^l)^\top \xi-{\xi^l}^\top{\rm diag}^{\bm d}(\bm\lambda)\xi^l\}=+\infty 
    \end{align*}
    if either $\mathcal Q- {\rm diag}^{\bm d}(\bm\lambda)\npreceq 0$ or $q+\textup{diag}^{\bm d}(\bm\lambda)\xi^l \notin \textup{range}(\mathcal Q- {\rm diag}^{\bm d}(\bm\lambda))$. Taking this into account and  introducing epigraphical variables, we can rewrite the dual problem as 
    \begin{align*}
        \inf_{\bm\lambda\succeq0}  &\; \langle \bm\lambda,\bm\epsilon\rangle+\sum_{l=1}^{M} \vartheta_l s_l  \\ 
        \text{s.t.} & \;\sup_{\xi\in\Xi}\{\xi^\top (\mathcal Q- {\rm diag}^{\bm d}(\bm\lambda))\xi+2(q+ {\rm diag}^{\bm d}(\bm\lambda)\xi^l)^\top \xi-{\xi^l}^\top {\rm diag}^{\bm d}(\bm\lambda)\xi^l\}\le s_l && \quad l\in[M] \\%
        & \hspace{15.4em} q+ \textup{diag}^{\bm d}(\bm\lambda)\xi^l\in\textup{range}(\mathcal Q- {\rm diag}^{\bm d}) && \quad l\in[M] \\
        & \hspace{22.7em} \mathcal Q- {\rm diag}^{\bm d}(\bm\lambda)\preceq0.  
    \end{align*}
    Since $\mathcal Q- {\rm diag}^{\bm d}(\bm\lambda)\preceq0$ and $q+ \textup{diag}^{\bm d}(\bm\lambda)\xi^l\in\textup{range}(\mathcal Q- {\rm diag}^{\bm d})$ for each $l\in[M]$, the quadratic program of the corresponding epigraphical  constraint admits an explicit solution \cite[page 653]{BS-VL:04}, which yields the equivalent constraint
    \begin{align*}
        (q+ {\rm diag}^{\bm d}(\bm\lambda)\xi^l)^\top  ( {\rm diag}^{\bm d}(\bm\lambda)-\mathcal Q)^{\dagger} (q+ {\rm diag}^{\bm d}(\bm\lambda)\xi^l)-{\xi^l}^\top  {\rm diag}^{\bm d}(\bm\lambda){\xi^l}-s_l\le0.
    \end{align*}
    Combining this with the characterization of positive semidefiniteness through the Schur complement \cite[page 651]{BS-VL:04}, we obtain the desired result. 
    %
    %
\end{proof}

\subsection{Proofs from Section~\ref{sec:DRCCP}}

Here we provide the proofs from Section~\ref{sec:DRCCP}. To this end, we generalize a saddle point result from \cite[Theorem 2.1]{AS-AK:02}, which enables the interchange of min and max operations for convex expected-value problems over general ambiguity sets of probability distributions. We first introduce some necessary preliminaries. Let $\mathcal A$ be a nonempty and convex set of probability distributions on the measurable space $(\Xi,\mathcal F)$. We also consider a closed convex subset $S$ of $\Rat{n}$, a convex neighborhood $V$ of $S$,  a function $\varphi:\Rat{n}\times\Xi\to\Rat{}$, and make the following assumption.

\begin{assumption} \label{assumption:saddle:point}
\longthmtitle{Saddle function \& domain} 
With $\mathcal A$, $S$, $V$, $\varphi$ as above we assume that:

\noindent (A1) For all $x\in V$ and $P\in\mathcal A$ the function $\xi\mapsto\varphi(x,\xi)$ is $\mathcal F$-measurable and $P$-integrable, i.e.,  
\begin{align*}
g(x,P):=\bE_P[\varphi(x,\xi)]\in\Rat{}.
\end{align*}

\noindent (A2) For all $\xi\in\Xi$ the function $x\mapsto\varphi(x,\xi)$ is convex on $V$.

\noindent (A3) For all $x\in V$ the max function 
\begin{align*}
f(x):=\sup_{P\in\mathcal A} \bE_P[\varphi(x,\xi)]\equiv\sup_{P\in\mathcal A}g(x,P)
\end{align*}
is finite valued, i.e., $f(x)<+\infty$.

Consider also a topology $\mathcal T$ on $\mathcal A$ and 
assume that:

\noindent (B1) The set $\mathcal A$ is sequentially compact.

\noindent (B2) For every $x\in V$ the function $P\mapsto g(x,P)$ is continuous. 
\end{assumption}

The following theorem summarizes the saddle-point property obtained by Propositions~2.1 and~2.2 in \cite{AS-AK:02} for our case where $\mathcal A$ is convex. 

\begin{thm}
\longthmtitle{Existence of saddle-point} 
Let Assumptions~\ref{assumption:saddle:point}(A1)-(A3) hold and $\bar x$ be an optimal solution of the problem
\begin{align} \label{primal:minmax}
\inf_{x\in S}\sup_{P\in\mathcal A}\bE_\mu[\varphi(x,\xi)]
\equiv\inf_{x\in S}\sup_{P\in\mathcal A}g(x,P)\equiv\inf_{x\in S}f(x).
\end{align}
Suppose further that 
\begin{align} \label{subdifferential:equality}
\partial f(\bar x)=\cup_{P\in\mathcal A^\star(\bar x)}\partial g_P(\bar x),
\end{align}
where 
$
\mathcal A^\star(\bar x):={\rm argmax}_{P\in\mathcal A}\,g(\bar x,P).
$ 
Then there exists $\bar P\in\mathcal A^\star(\bar x)$ such that $(\bar x,\bar P)$ is a saddle point of \eqref{primal:minmax}, namely 
\begin{align} \label{saddle:minmax}
\inf_{x\in S}\sup_{P\in\mathcal A}g(\bar x,\bar P)=\sup_{P\in\mathcal A}\inf_{x\in S}g(\bar x,\bar P).
\end{align}
\end{thm}

Using this result enables us to generalize \cite[Theorem 2.1]{AS-AK:02} in terms of the growth of the objective function for minimax DRO problems over MTHs.

\begin{thm} \label{thm:stochastic:maxmin}
\longthmtitle{Stochastic min-max equality} 
Let Assumption~\ref{assumption:saddle:point} hold, i.e., all (A1)-(A3) and (B1), (B2) hold, and assume that $\bar x$ is an optimal solution of \eqref{primal:minmax}. Then \eqref{primal:minmax} also admits a saddle point $(\bar x,\bar P)$ and the minimax equality \eqref{saddle:minmax} holds.    
\end{thm}

\begin{proof}
Following \cite{AS-AK:02}, we first note that by \cite[Theorem 3, Page 201]{ADI-VMT:79}, it holds that 
\begin{align*}
\partial f(\bar x)={\rm cl}\big(\cup_{P\in\mathcal A^\star(\bar x)}\partial g_P(\bar x)\big).
\end{align*}
Therefore, to establish \eqref{subdifferential:equality} and conclude the proof, it suffices to show that $\cup_{P\in\mathcal A^\star(\bar x)}\partial g_P(\bar x)$ is closed. To this end, let $z_k\to z\in\Rat{n}$ with each $z_k\in\partial g_{P_k}(\bar x)$ for some $P_k\in\mathcal A^\star(\bar x)$. Since $\{P_k\}\subset\mathcal A$ and $\mathcal A$ is sequentially compact by (B1), a subsequence converges to some $P_\star\in\mathcal A$.
With a slight abuse of notation, we also index this sequence by $k$.
Taking further into account that $P\mapsto g_P(\bar x)$ is continuous by (B2), we have 
that $g_{P_k}(\bar x)\to g_{P_\star}(\bar x)$. By the fact that each $P_k\in\mathcal A^\star(\bar x)$ and the definition of $\mathcal A^\star$, we get that also $P_\star\in\mathcal A^\star(\bar x)$. Thus, it suffices to show that $z\in\partial g_{P_\star}(\bar x)$. Since $z_k\in\partial g_{P_k}(\bar x)$, we have for each $x\in V$ that 
\begin{align*}
g_{P_k}(x)-g_{P_k}(\bar x) \ge\langle z_k,x-\bar x\rangle  \Longrightarrow 
g_{P_\star}(x)-g_{P_\star}(\bar x) \ge\langle z,x-\bar x\rangle,  
\end{align*}
where we exploited again continuity of $g$ with respect to $P$. The proof is now complete. 
\end{proof}

We next formulate a result that enables us to apply this min-max interchange to DRO problems over MTHs. 

\begin{corollary} \label{corollary:minmax}
\longthmtitle{Stochastic min-max equality for MTHs} 
Consider the Polish space $\Xi$ with its Borel $\sigma$-algebra, let $\mathcal A\equiv\mathcal T_p(Q,\bm\varepsilon)$, $\mathcal T$ be the weak topology on $\mathcal A$, and assume that there exists $r\in[0,p)$, such that for each $x\in V$, the function $\xi\mapsto \varphi(x,\xi)$ is continuous and belongs to the class $\zeta\in \mathcal G_r$. Let also Assumptions~\ref{assumption:saddle:point}(A1)-(A3) hold and $\bar x$ be an optimal solution of \eqref{primal:minmax}. Then \eqref{primal:minmax} admits a saddle point $(\bar x,\bar P)$ and the minimax equality \eqref{saddle:minmax} holds.    
\end{corollary}

\begin{proof}
From Theorem~\ref{thm:stochastic:maxmin} it suffices to establish (B1) and (B2), namely, that $\mathcal A$ is weakly (sequentially) compact and that each function $P\mapsto g(x,P)$ is continuous. These properties follow directly from Theorems~\ref{thm:weak:compactness} and~\ref{thm:existence:inner:maximization}(ii), respectively. 
\end{proof}

Using this min-max interchange property, we can prove the results of Section~\ref{sec:DRCCP} that yield tractable reformulations of distributionally robust chance-constrained problems.  

\begin{proof}[Proof of Lemma \ref{lemma:mini-max}]
The proof follows by using Corollary~\ref{corollary:minmax}, weak compactness of $\mathcal T_p(\bm P_\xi^N,\bm\epsilon)$, and the exact same arguments as in the proof of \cite[Lemma IV.2.]{HRA-CA-JL:18}.
\end{proof}

\begin{proof}[Proof of Proposition \ref{prop:reformulations:CVaR:convex:functions}]
  The proof is analogous to the proof of \cite[Proposition IV.3]{HRA-CA-JL:18}. From Lemma \ref{lemma:mini-max}(i), Assumption~\ref{assumption:convex:concave}(ii), and the strong duality result of Theorem \ref{prop:dual}, we get that
    \begin{align*}
        \sup_{P\in \mathcal T_p(Q,\bm\varepsilon)} & \inf_{\tau\in\mathbb R}\big\{\mathbb E_{P}[(f(x,\xi)+\tau)_+]-\tau\alpha\big\} \\
         = & \inf_{\tau\in\mathbb R} \inf_{\bm\lambda\succeq0}\Big\{ \langle\boldsymbol{\lambda, \varepsilon}\rangle-\tau\alpha+ \sum_{l\in[M]}\vartheta_l\sup_{\xi\in\Xi}\Big\{(f(x,\xi)+\tau)_+-\sum_{k=1}^n \lambda_k\|\xi_k^l-\xi_k\|^p\Big\}\Big\}.
    \end{align*}
    The inner infimum with respect to $\bm\lambda\succeq0$ is attained by Theorem \ref{prop:dual} and the infimum with respect to $\tau$ is also attained because of Lemma \ref{lemma:mini-max}(ii). Thus, introducing epigraphical variables $s_l\in\mathbb R$, $l\in[M]$, the feasible set \eqref{eq:DR:CVaR:constraints:1} comprises the points $x\in\mathcal X$ for which  
    \begin{align*}
         & \langle\boldsymbol{\lambda, \varepsilon}\rangle+ \sum_{l=1}^{M}\vartheta_l s_l\le \tau\alpha \nonumber \\
      \quad & \sup_{\xi\in\Xi}\Big\{(f(x,\xi)+\tau)_+-\sum_{k=1}^n\lambda_k \|\xi_k-{\xi}_k^l\|^p\Big\}\le s_l  \quad  l\in[M] 
    \end{align*}
    for some $\bm\lambda\succeq 0$ and   $s_l\in\Rat{}$, where $l\in[M]$. Using the same arguments as in the final part of the proof of \cite[Proposition IV.3]{HRA-CA-JL:18}, this constraint is equivalent to
    \begin{align*}
        \Big(\sup_{\xi\in\Xi}\Big\{(f(x,\xi)+\tau)_+-\sum_{k=1}^n\lambda_k \|\xi_k-{\xi}_k^l\|^p\Big\}\Big)_+\le s_l \quad l\in[M],
    \end{align*}
    which yields the desired result.
\end{proof}

\begin{proof}[Proof of Proposition~\ref{prop:CVaR:piecewise:affine:functions}]

     Since $f(x,\xi)$ is affine in $\xi$ and convex in $x$, and $\Xi$ is a compact polyhedral set, it satisfies Assumption \ref{assumption:convex:concave}. Therefore, the feasible set of \eqref{eq:DRCCP:CVaR} can be determined by Proposition \ref{prop:reformulations:CVaR:convex:functions}. Using the same arguments as in the proof of \cite[Proposition V.1]{HRA-CA-JL:18}, the constraints involving the auxiliary variables $s_l$ are equivalently written
    \begin{align}
      s_l \ge \Big(b_j(x)+\tau+ \sup_{\xi\in\Xi}\Big\{\langle x, A_j\xi\rangle-\sum_{k=1}^n\lambda_k \|\xi_k-{\xi}_k^l\|]\Big\}\Big)_{+} \quad l\in[M],\; j\in[m]. \label{eq: development 22}
    \end{align}
    By adjusting the derivations in \cite[equation (23)]{HRA-CA-JL:18} to these constraints and considering for each $l\in[M]$ and $j\in[m]$ the dual variable $z_{lj}\in\Rat{d}$, we get
    \begin{align*}
        \sup_{\xi\in\Xi}\Big\{\langle x, A_j\xi\rangle-\sum_{k=1}^n\lambda_k \|\xi_k-\xi_k^l\|\Big\} 
        &\overset{(a)}{=} \sup_{\xi\in\Xi}\Big\{\langle x, A_j\xi \rangle-\sum_{k=1}^n \sup_{\|\textup{pr}_k^{\bm d}(z_{lj})\|_{*}\le\lambda_k} \langle \textup{pr}_k^{\bm d}(z_{lj}), \xi_k-\xi_k^l\rangle\Big\}, \\
        &\overset{(b)}{=} \inf_{\|\textup{pr}_k^{\bm d}(z_{lj})\|_{*}\le\lambda_k,\; k\in[n]}\Big\{ \langle z_{lj}, {\xi}^l\rangle+ \sup_{\xi\in\Xi}\{\langle A_j^\top x-z_{lj},\xi\rangle\}\Big\}, \\
        &\overset{(c)}{=}  \inf_{\|\textup{pr}_k^{\bm d}(z_{lj})\|_{*}\le\lambda_k,\;k\in[n]}\Big\{ \langle z_{lj}, \xi^l\rangle + \inf_{\underset{\eta_{lj}\succeq0}{z_{lj}= A_j^\top x-C^\top \eta_{lj}}} \langle \eta_{lj}, h\rangle \Big\}, \\
        &= \inf_{ \underset{\eta_{lj}\succeq0,\;z_{lj}=A_j^\top x-C^\top\eta_{lj}}{\|\textup{pr}_k^{\bm d}(z_{lj})\|_{*}\le\lambda_k,\;k\in[n]}}\big\{ \langle z_{lj}, \xi^l\rangle+\langle \eta_{lj}, h\rangle\big\}, \\
        &= \inf_{\underset{k\in[n],\;\eta_{lj}\succeq0}{\|\textup{pr}_k^{\bm d}(A_j^\top x-C^\top{\eta_{lj}})\|_{*}\le\lambda_k}}\big\{\langle A_j^\top x - C^\top\eta_{lj}, \xi^l\rangle+ \langle {\eta_{lj}}, h\rangle\big\}, 
    \end{align*}
    Here (a) follows from the definition of the dual norm and (c) from linear programming duality for $\sup_{\xi\in\Xi}\langle A_j^\top x-z_{lj}, \xi\rangle$. Since the supremum in (b) is taken over the compact set $\Xi$, it is also attained. Thus, by linear programming duality, the inner infimum in (c) is also attained for some $\eta_{lj}\succeq0$. From these derivations, we can equivalently rewrite for each $j\in[m]$ and $l\in[M]$ the constraint \eqref{eq: development 22} as
    \begin{align*}
        s_l\ge \Big(b_j(x)+\tau+ \inf_{\underset{ \|\textup{pr}_k^{\bm d}(A_j^\top x- C^\top {\eta_{lj}})\|_{*}\le\lambda_k}{\eta_{lj}\ge0}}\big\{\langle A_j^\top x - C^\top\eta_{lj}, \xi^l\rangle+\langle{\eta_{lj}}, h\rangle \big\}\Big)_+.
    \end{align*}
    Recalling that the infimum in this expression is attained for some $\eta_{lj}\succeq0$, the constraint is satisfied if and only if there exist $s_l\ge0$ and $\eta_{lj}\succeq0$ such that  
    \begin{align*}
        b_j(x)+\tau+\langle A_j^\top x-C^\top \eta_{lj},  {\xi}^l\rangle+\langle{\eta_{lj}}, h\rangle & \le s_l && \quad l\in[M],\; j\in[m], \\
        \left\|\textup{pr}_k^{\bm d} (A_j^\top x-  C^\top {\eta_{lj}})\right\|_{*} & \le \lambda_k &&     \quad l\in[M],\; j\in[m],\; k\in[n].
    \end{align*}
    This yields the reformulation given in the statement and concludes the proof.
\end{proof}

\section{Shrinkage of MTHs with clustered reference  distribution} \label{sec:appendix:clustering}

Here, we quantify how the clustering process affects the size of data-driven ambiguity sets under the requirement that they contain the true distribution with prescribed probability. In particular, we seek to determine how the rate by which they shrink with the number of samples can be associated with the number of clusters when these are also allowed to depend on the sample size. This is motivated by the fact that MTHs shrink at favorable rates compared to Wasserstein balls and we want to establish to which degree such rates are retained under the clustering. To this end, we use concentration-of-measure results that yield confidence bounds on the Wasserstein distance between compactly supported distributions and their empirical approximations. 

\begin{prop} 
\label{prop:W:distance:bound}
\longthmtitle{Distance between true and empirical distribution \cite[Proposition~4.2]{DB-JC-SM:24-tac}} 
Assume $P_\xi$ is supported on $\Xi\subset\mathbb R^d$ with $\rho_\Xi := 1/2 {\rm diam}(\Xi)<\infty$ and that $d>2p$. Consider the empirical distribution $P_\xi^N$ inferred from $N$ i.i.d. samples of $P_\xi$. Then
\begin{align*} 
    \mathbb P\big({\small W_p(P^K,\bm P_\xi^N)\le \rho_\Xi c_1(d,p) (c_2(d,p)+(\ln \beta^{-1})^{1/2p}){N}^{-1/d}} \big)\ge 1-\beta,
\end{align*}
where $c_1$, $c_2$ are positive constants and $c_1$ also depends on the norm on $\Rat{d}$.
\end{prop}
%


Exploiting this result, we determine an upper bound for the Wasserstein distance between the clustered distribution $\widehat P_\xi$ and the product empirical distribution $\bm P_\xi^N$. 
\begin{prop} \longthmtitle{Wasserstein distance from optimal quantizer} \label{prop:W:Phat:bmPxiN}
    Assume $\widehat P_\xi$ is an optimal quantizer of $\bm P_\xi^N$ supported at $K$ atoms with 
    \begin{align}
        \log_N K = q \label{eq:assumption:clustered:distribution:K}
    \end{align}
    for some $q>1$ and let $d>2p$. Then, we have  
    \begin{align*} 
        W_p(\widehat P_\xi,\bm P_\xi^N) \le C(\rho_\Xi,d,p)N^{-q/d},
    \end{align*}
    with $C(\rho_\Xi,d,p):=\rho_\Xi c_1(d,p)c_2(d,p)$ and $\rho_\Xi$, $c_1$, $c_2$ as given in Proposition~\ref{prop:W:distance:bound}.
\end{prop}

\begin{proof}
First, fix a realization of the product empirical distribution $\bm P_\xi^N\in\mathcal P^{N^n}(\Xi)$ and let $P^K\in\mathcal P^K(\Xi)$ be supported on $K$ atoms of $\bm P_\xi^N$. Then from Proposition~\ref{prop:W:distance:bound}, the probabilistic argument suggests that we can select for any $\beta\in(0,1)$ a distribution $P^K$ such that 
\begin{align*}
W_p(P^K, \bm P_\xi^N) & \le \rho_\Xi c_1(d,p) (c_2(d,p)+(\ln \beta^{-1})^{1/2p})K^{-1/d} \\
& = \rho_\Xi c_1(d,p) (c_2(d,p)+(\ln \beta^{-1})^{1/2p})N^{-q/d},
\end{align*} 
where we used \eqref{eq:assumption:clustered:distribution:K} to obtain the second line. Thus, we can consider a sequence of probabilities $\beta_i\nearrow1$ and discrete distributions $P_i^K$ with 
\begin{align*}
W_p^p(P_i^K, \bm P_\xi^N)\le \eps_i:=\rho_\Xi c_1(d,p)(c_2(d,p)+(\ln \beta_i^{-1})^{1/2p}) N^{-q/d}.
\end{align*} 
Identifying each discrete distributions $P_i^K$ with a finite sequence of points in $\Xi^K \subset \mathbb R^{Kd}$, we deduce by compactness of $\Xi^K$ that a subsequence of $P_i^K$ converges to some $P_{\star}^K$ with
\begin{align*}
W_p(P_{\star}^K, \bm P_\xi^N)\le \lim_{i\to\infty}\eps_i= C(\rho_\Xi,d,p) N^{-q/d},
\end{align*}
where $C$ is given in the statement. Since $\widehat P_\xi$ is an optimal quantizer of $\bm P_\xi^N$ consisting of $K$ atoms, which implies that $W_p(\widehat P_\xi, \bm P_\xi^N) \le W_p(P_\star^K, \bm P_\xi^N)$, we obtain the desired result. 
\end{proof}
\begin{rem} \longthmtitle{Quantization quality of the clustered distribution} The optimal quantization problem is non-convex and its solution using Lloyd-type methods only guarantees convergence to some local optimum. However, this often yields satisfactory, i.e., near-optimal approximations of the global optimum~\cite{RO-YR-LJS-CS:13, DA-SV:06, BB-BM-AV-RK-SV:12}. As a result, a bound close to the one derived in Proposition \ref{prop:W:Phat:bmPxiN} is likely to hold in practice.
\end{rem}

The following variant of Proposition~\ref{prop:W:distance:bound} establishes that MTHs built from random variables with multiple independent components that contain the true distribution with prescribed confidence shrink at favorable rates with the number of samples compared to Wasserstein balls. 

\begin{prop} 
\label{prop:size:yellow:ball}
\longthmtitle{Ambiguity hyperrectangle size dependence on data \cite[Proposition 5.2]{LMC-DB-TO:22}.} 
Let $P_\xi:=P_{\xi_1}\otimes\ldots\otimes P_{\xi_n}$, where $\Xi_k$ is a compact subset of $\mathbb R^{d_k}$ and $d_k>2p$ for each $k\in[n]$. 
Then one can select the radii $\bm\eps=(\eps_1,\ldots,\eps_n)$ so that 
 \begin{align*}
  \mathbb P( P_\xi\in\mathcal T_p(\bm P_\xi^N, \bm\eps)) \ge 1-\beta
\quad \text{and}\quad
  \mathcal T_p(\bm P_\xi^N,\bm\eps)\subset \mathcal B_p(\bm P_\xi^N,\eps),
    \end{align*}
    where
    \begin{align} \label{eq:eps:yellow:bound}
     \eps=\widehat C(\rho_\Xi, \beta, n, d, p) N^{-1/d_{\rm max}},    
    \end{align}
    $d_{\rm max}:= \max_{k\in[n]}d_k$, and $\widehat C$ also depends on the norm in $\Rat{d}$.
\end{prop}

Combining Propositions~\ref{prop:inflation:containment},~\ref{prop:W:Phat:bmPxiN}, and~\ref{prop:size:yellow:ball}, we can determine how a shifted MTH that is centered at a quantizer of the product empirical distribution shrinks with the number of samples. 

\begin{corollary}
 \longthmtitle{Size of inflated MTHs}
 With $P_\xi$ as in Proposition~\ref{prop:size:yellow:ball}, let $\widehat P_\xi$ be an optimal quantizer of $\bm P_\xi^N$ supported at $K$ atoms and assume that  \eqref{eq:assumption:clustered:distribution:K} holds for some $q>1$. Consider also an optimal transport plan $\pi$  for the $W_p$ distance between $\widehat P_\xi$ and $\bm P_\xi^N$. Then with $\bm\eps:=(\eps_1,\ldots,\eps_n)$ as in Proposition~\ref{prop:size:yellow:ball} and $\bm\epsilon:=(\epsilon_1,\ldots,\epsilon_n)$ where $\epsilon_k^p:=\int_{\Xi \times \Xi} \|\eta_k-\zeta_k\|^p d\pi(\eta,\zeta)$ for each $k\in[n]$, we have 
    \begin{align}
    \label{thm:small:shifted:rect:conf}
        \mathbb P\big(P_\xi\in \mathcal T_p(\widehat P_\xi, \bm\eps+\bm\epsilon)\big)\ge1-\beta 
    \quad\text{and}\quad 
         \mathcal T_p(\widehat P_\xi,\bm\eps+\bm\epsilon)\subset \mathcal{B}_p(\widehat P_\xi, \eps^\star),
    \end{align}
   where
    \begin{align*}
        \eps^\star:= C(\rho_\Xi,d,p) N^{-q/d}+ \widehat C(\rho_\Xi, \beta, n, d, p) N^{-1/d_{\rm max}} 
    \end{align*}
  and $C$ and $\widehat C$, $d_{\rm max}$ are given in Proposition~\ref{prop:W:Phat:bmPxiN} and~\ref{prop:size:yellow:ball}, respectively.
\end{corollary}
\begin{proof}
      From Proposition \ref{prop:size:yellow:ball} we select $\bm\eps$ so that $\mathbb P\big( \mathcal T_p(\bm P_\xi^N, \bm\eps) \big)\ge 1-\beta$. Since also $\mathcal T_p(\bm P_\xi^N,\bm\eps)\subset \mathcal T_p(\widehat P_\xi, \bm\eps+\bm\epsilon)$ by Proposition \ref{prop:inflation:containment}(i), we deduce the inequality in \eqref{thm:small:shifted:rect:conf}. The inclusion in \eqref{thm:small:shifted:rect:conf} follows directly by Propositions~\ref{prop:W:Phat:bmPxiN} and~\ref{prop:size:yellow:ball}, the definition of $\eps^\star$, and the triangle inequality
    \begin{align*}
        W_p(\widehat P_\xi,P_\xi) \le W_p(\widehat P_\xi, \bm P_\xi^N) + W_p(\bm P_\xi^N, P_\xi) 
    \end{align*}
    for the Wasserstein distance. 
    \end{proof}

This corollary suggests that choosing the number $K\equiv K(N)$ of clusters, or equivalently the exponent $q$ so that $1/d<q/d\le 1/d_{\max}$, provides a tunable tradeoff between the size of the MTH and the complexity of its reference with fixed probabilistic guarantees. We also note that by clustering the product empirical distribution, we automatically obtain a (suboptimal) transport plan that can be used to compute explicit upper bounds for $\epsilon_1,\ldots,\epsilon_n$.

\medskip
\bibliography{alias,references} 
\bibliographystyle{IEEEtranS}
\end{document}